\documentclass[12pt,reqno]{amsart}
\usepackage[margin=1in]{geometry}
\usepackage{amsmath,amssymb,amsthm,graphicx,amsxtra, setspace}
\usepackage[utf8]{inputenc}
\usepackage{mathrsfs}
\usepackage{hyperref}
\usepackage{upgreek}
\usepackage{mathtools}
\allowdisplaybreaks

\usepackage[pagewise]{lineno}

\newtheorem{theorem}{Theorem}[section]
\newtheorem{lemma}[theorem]{Lemma}
\newtheorem{proposition}[theorem]{Proposition}
\newtheorem{assumption}[theorem]{Assumption}
\newtheorem{corollary}[theorem]{Corollary}
\newtheorem{definition}[theorem]{Definition}

\newtheorem{remark}[theorem]{Remark}

\let\originalleft\left
\let\originalright\right
\renewcommand{\left}{\mathopen{}\mathclose\bgroup\originalleft}
\renewcommand{\right}{\aftergroup\egroup\originalright}

\renewcommand{\d}{\/\mathrm{d}\/}

\def\w{\textbf{W}^{\varepsilon}_{{\theta}^{\varepsilon}}}

\def\e{\varepsilon}

\def\S{\mathrm{S}}

\def\L{\mathbb{L}}
\def\A{\mathrm{A}}

\def\C{\mathrm{C}}
\def\f{\mathbf{f}}

\def\B{\mathrm{B}}
\def\D{\mathrm{D}}
\def\y{\mathbf{y}}

\def\E{\mathbb{E}}

\def\x{\mathbf{x}}

\def\p{\mathbf{p}}

\def\z{\mathbf{z}}
\def\v{\mathbf{v}}
\def\V{\mathbb{v}}
\def\w{\mathbf{w}}

\def\G{\mathrm{G}}
\def\Q{\mathrm{Q}}

\def\N{\mathbb{N}}

\def\V{\mathbb{V}}
\def\wi{\widetilde}
\def\Q{\mathrm{Q}}
\def\u{\mathrm{U}}
\def\P{\mathrm{P}}
\def\u{\mathbf{u}}
\def\H{\mathbb{H}}
\def\n{\mathbf{n}}

\newcommand{\R}{\mathbb{R}}

\renewcommand{\d}{\/\mathrm{d}\/}

\newcommand{\Addresses}{{
		\footnote{
			\noindent \textsuperscript{1,2}Department of Mathematics, Indian Institute of Technology Roorkee-IIT Roorkee,
			Haridwar Highway, Roorkee, Uttarakhand 247667, INDIA.\par\nopagebreak
			\noindent  \textit{e-mail:} \texttt{Manil T. Mohan: maniltmohan@ma.iitr.ac.in, maniltmohan@gmail.com.}
			
			\textit{e-mail:} \texttt{Kush Kinra: kkinra@ma.iitr.ac.in.}
			
			\noindent \textsuperscript{*}Corresponding author.
			
			\textit{Key words:} Stochastic convective Brinkman-Forchheimer equations, cylindrical Wiener process, random dynamical system, random attractors, flattening property,  upper semicontinuity.
			
			Mathematics Subject Classification (2020): Primary 35B41, 35Q35; Secondary 37L55, 37N10, 35R60.

}}}

\begin{document}
	
	\title[Random attractors for SCBF equations]{Existence and upper semicontinuity of random attractors for the 2D stochastic convective Brinkman-Forchheimer equations in bounded domains
		\Addresses}
	
	\author[K. Kinra and M. T. Mohan]
	{Kush Kinra\textsuperscript{1} and Manil T. Mohan\textsuperscript{2*}}

	\maketitle
	
	\begin{abstract}
		In this work, we discuss the large time behavior of the solutions of the two dimensional stochastic convective Brinkman-Forchheimer (SCBF) equations in bounded domains. Under the functional setting $\V\hookrightarrow\H\hookrightarrow\V'$, where $\H$ and $\V$ are appropriate separable Hilbert spaces and the embedding $\V\hookrightarrow\H$ is compact, we establish the existence of random attractors in $\H$ for the stochastic flow generated by the 2D SCBF equations perturbed by small additive noise. We prove the upper semicontinuity of the random attractors for the 2D SCBF equations in $\H$, when the coefficient of random term approaches zero. Moreover, we obtain the existence of random attractors in a more regular space $\V$, using the pullback flattening property. The existence of random attractors ensures the existence of invariant compact random set and hence we show the existence of an invariant measure for the 2D SCBF equations. 
	\end{abstract}

	\section{Introduction} \label{sec1}\setcounter{equation}{0}
	The study of the asymptotic behavior of dynamical systems is one of the most important areas of mathematical physics. A comprehensive investigation on the  attractors for the deterministic infinite dimensional dynamical systems has been carried out in  \cite{ICh,JCR,R.Temam}, etc. However, their corresponding stochastic versions also have great importance, therefore the analysis of the infinite dimensional random dynamical system (RDS) is also a predominant branch of stochastic partial differential equations (SPDEs). A detailed study as well as an elaborate literature is available in \cite{Arnold} on the generation of random dynamical systems for  stochastic ordinary differential equations and SPDEs. The existence of random attractors for a large class of SPDEs like stochastic reaction-diffusion equations, the stochastic $p$-Laplace equation and stochastic porous media equations, etc  driven by general additive noise is established in \cite{BGWL}. As far as the stochastic Navier-Stokes equations (SNSE) are concerned, the notion of random attractors was introduced in \cite{BCF,CDF,CF}, etc and the authors established the existence of random attractors for the 2D SNSE in bounded domains. Since the generation of RDS and random attractors for the Navier-Stokes equations (NSE) is a very active  area of research, we restrict ourselves to those works which are relevant to the results of this paper.  The existence of random attractors for several physically relevant stochastic models is proved in the works \cite{BGT, BLL, BLW, BCLLLR, Crauel1, KM, LG, FLYB, You} etc, and the references therein. 
	
	In the functional setting $\V\hookrightarrow\H\hookrightarrow\V'$, where $\H$ and $\V$ are appropriate separable Hilbert spaces (see section \ref{sec2} for details on the function spaces),  since we do not have any method to find an absorbing set in a more regular space than $\V$, we are not able to prove the existence of random attractors in $\V$ using compactness arguments. To resolve this problem, the authors in \cite{KL} introduced a method to find the existence of random attractors using the pullback flattening property and this method  became successful to prove the existence of random attractors for the 2D SNSE as well as stochastic reaction-diffusion equations in $\V$. The authors in the works  \cite{FY, KLR,  Zhao, Zhao1}, etc obtained the existence of random attractors  in $\V$ for different stochastic models appearing in fluid mechanics by  verifying the pullback flattening property. The authour in \cite{You} proved the existence of a random attractors for the stochastic 3D damped NSE in bounded domains with additive noise by verifying the pullback flattening property. It appears to us that the results obtained in the work \cite{You} may not hold true in bounded domains due to the technical difficulties discussed in the works \cite{KZ,MTM}, etc (commutativity of the projection operator with $-\Delta$ and the nonzero boundary condition of the projected nonlinear damping term). 
	
	Furthermore, in the study of random attractors, one more property of the random attractors was introduced in \cite{CLR}, which is the upper semicontinuity of random  attractors. Roughly speaking, if $\mathcal{A}$ is a global attractor for the deterministic system and $\mathcal{A}_\varepsilon$ is a random attractor for the corresponding stochastic system perturbed by a small noise, we say that these attractors have the property of upper semicontinuity if $$\lim\limits_{\varepsilon\to 0}d(\mathcal{A}_\varepsilon, \mathcal{A})=0,$$ where $d$ is the Hausdorff semidistance given by $d(A, B)=\sup\limits_{y\in A}\inf\limits_{z\in B}\rho(y,z),$ for any $A, B\subset X,$ on a Polish space $(X,\rho)$. After introducing the concept of upper semicontinuity, the authors in \cite{CLR} proved the upper semicontinuity of random attractors for the 2D SNSE and stochastic reaction-diffusion equations. The existence and the upper semicontinuity of a pullback attractor for the stochastic retarded 2D NSE on a bounded domain are obtained in \cite{XJXD}. The author in  \cite{BW} established the upper semicontinuity of random attractors for non-compact random dynamical systems and  he applied  this result to a stochastic reaction-diffusion equation on the whole space. The upper semicontinuity of random attractors for the stochastic $p$-Laplacian equations on unbounded domains is obtained in \cite{JYH}. 
	
	Our main aim of this article is to study the asymptotic behavior of solutions of the stochastic version of the following system perturbed by small additive noise. Let $\mathcal{O} \subset \R^2$ be a bounded domain with smooth boundary $\partial\mathcal{O}$ and consider the following convective Brinkman-Forchheimer (CBF) equations in $\mathcal{O}$ with homogeneous Dirichlet boundary conditions:
	\begin{equation}\label{1}
	\left\{
	\begin{aligned}
	\frac{\partial \u}{\partial t}-\mu \Delta\u+(\u\cdot\nabla)\u+\alpha\u+\beta|\u|^{r-1}\u+\nabla \p&=\mathbf{f}, \ \text{ in } \ \mathcal{O}\times(0,\infty), \\ \nabla\cdot\u&=0, \ \text{ in } \ \mathcal{O}\times(0,\infty), \\
	\u&=\mathbf{0}\ \ \text{ on } \ \partial\mathcal{O}\times(0,\infty), \\
	\u(0)&=\u_0 \ \text{ in } \ \mathcal{O},\\
	\int_{\mathcal{O}}p(x,t)\d x&=0, \ \text{ in } \ (0,\infty).
	\end{aligned}
	\right.
	\end{equation}
	The convective Brinkman-Forchheimer equations \eqref{1} describe the motion of incompressible fluid flows in a saturated porous medium.  Here $\u(t , x) \in \R^2$ represents the velocity field at time $t$ and position $x$, $p(t,x)\in\R$ denotes the pressure field, $\f(t,x)\in\R^2$ is an external forcing. The final condition in \eqref{1} is imposed for the uniqueness of the pressure $p$. The constant $\mu$ represents the positive Brinkman coefficient (effective viscosity), the positive constants $\alpha$ and $\beta$ represent the Darcy (permeability of porous medium) and Forchheimer (proportional to the porosity of the material) coefficients, respectively. The absorption exponent $r\in[1,\infty)$ and  $r=3$ is known as the critical exponent. For $\alpha=\beta=0$, we obtain the classical 2D NSE. Thus, one can consider the system \eqref{1} as a damped Navier-Stokes equations with the linear and nonlinear damping terms $\alpha\u$ and $\beta|\u|^{r-1}\u$, respectively. 
	
	The global solvability of the system \eqref{1} in two and three dimensional bounded domains is available in \cite{SNA,KZ,MTM0}, etc. The existence of global attractors for the 2D deterministic CBF equations in $\H$ and $\V$ on Poincar\'e domains (may be unbounded) is proved in \cite{MTM1,MTM2}, respectively. For a sample  literature on the  attractors for two and three dimensional CBF equations and damped Navier-Stokes equations, the interested readers are referred to see \cite{KZ,KWH,AAI,AIKP,MTM1} etc. On the stochastic counterpart, the existence of a unique pathwise strong solution to the two and three dimensional stochastic convective Brinkman-Forchheimer (SCBF) equations  (see \eqref{S-CBF} below)  is obtained in \cite{MTM}. The authors in \cite{KM} proved the existence of random attractors in $\H$ for the 2D SCBF equations perturbed by additive rough noise on Poincar\'e domains. The existence of random dynamical systems and random attractors in $\H$ for a large class of locally monotone SPDEs  perturbed by additive L\'evy noise is obtained in \cite{GLS}. In this work, we establish the existence of a random attractor in $\H$ for the 2D SCBF equations for $r\in[1,\infty)$ in bounded domains. Furthermore, we consider the stability of global attractor and prove that the random attractors for the 2D SCBF system with small additive noise will converge to the global attractor of the unperturbed 2D CBF system, when the parameter of the perturbation  tends to zero. It is remarked that under smoothness assumptions on the noise, the upper semicontinuity property in $\H$ holds true on Poincar\'e domains also. We also point out that for $r>3,$ our model does not fall into the SPDEs explained in \cite{GLS} and for $r\in[1,\infty)$, we obtain the random attractors in $\V$.

	The rest of the paper  is organized as follows. In the next section, we define the  functional spaces which are needed for the global solvability of the system \eqref{1}. We introduce the linear and nonlinear operators also in the same section along with their properties. Furthermore, we provide the definitions and results on random dynamical systems and  random attractors. The 2D SCBF equations is also considered in the same section  and we discuss about its global solvability results. The metric dynamical system as well as the  random dynamical system for our model is constructed in the section \ref{sec3}. The existence of random attractors in $\H$ for the 2D SCBF equations is proved in section \ref{sec4} by establishing absorbing balls in $\H$ and $\V$, and then using compactness arguments provided in Theorem 3.11, \cite{CF} (Theorems \ref{H_ab}, \ref{V_ab}, \ref{Main_Theoem_H}). In section \ref{sec5}, we establish the upper semicontinuity of the random attractors in $\H$ (Theorem \ref{USC}). We also remark that such results holds true even in Poincar\'e domains (Remark \ref{rem5.2}). The existence of random attractors in $\V$ for the 2D SCBF equations by using pullback flattening property is proved in section \ref{sec6} (Theorems \ref{flattning} and \ref{Main_Theoem_V}). Note that the existence of random attractors ensures the existence of invariant compact random set and hence we show the existence of an invariant measure for the 2D SCBF equations  in section \ref{sec7} (Theorem \ref{thm7.3}) by invoking Corollary 4.4, \cite{CF}.

	\section{Mathematical Formulation and Preliminaries}\label{sec2}\setcounter{equation}{0}
	In this section, we present the necessary function spaces needed to obtain the existence and uniqueness of solutions as well as asymptotic behavior of solutions of the system \eqref{1}.  In our analysis, the parameter $\alpha$ does not play a major role, therefore we set $\alpha$ to be zero in \eqref{1} for the rest of our work.
	\subsection{Function spaces} We denote $\C_0^{\infty}(\mathcal{O};\R^2)$ for the space of all infinitely differentiable functions  ($\R^2$-valued) with compact support in $\mathcal{O}\subset\R^2$.  Let us define 
	\begin{align*} 
	\mathcal{V}&:=\{\u\in\C_0^{\infty}(\mathcal{O},\R^2):\nabla\cdot\u=0\},\\
	\mathbb{H}&:=\text{the closure of }\ \mathcal{V} \ \text{ in the Lebesgue space } \L^2(\mathcal{O})=\mathrm{L}^2(\mathcal{O};\R^2),\\
	\mathbb{V}&:=\text{the closure of }\ \mathcal{V} \ \text{ in the Sobolev space } \H_0^1(\mathcal{O})=\mathrm{H}_0^1(\mathcal{O};\R^2),\\
	\widetilde{\L}^{p}&:=\text{the closure of }\ \mathcal{V} \ \text{ in the Lebesgue space } \L^p(\mathcal{O})=\mathrm{L}^p(\mathcal{O};\R^2),
	\end{align*}
	for $p\in(2,\infty)$. Then under some smoothness assumptions on the boundary, we characterize the spaces $\H$, $\V$ and $\widetilde{\L}^p$ as 
	$
	\H=\{\u\in\L^2(\mathcal{O}):\nabla\cdot\u=0,\u\cdot\mathbf{n}\big|_{\partial\mathcal{O}}=0\}$,  with the norm  $\|\u\|_{\H}^2:=\int_{\mathcal{O}}|\u(x)|^2\d x,
	$
	where $\mathbf{n}$ is the outward normal to $\partial\mathcal{O}$, and $\u\cdot\n\big|_{\partial\mathcal{O}}$ should be understood in the sense of trace in $\H^{-1/2}(\partial\mathcal{O})$ (cf. Theorem 1.2, Chapter 1, \cite{Temam}), 
	$
	\V=\{\u\in\H_0^1(\mathcal{O}):\nabla\cdot\u=0\},$  with the norm $ \|\u\|_{\V}^2:=\int_{\mathcal{O}}|\nabla\u(x)|^2\d x,
	$ and $\widetilde{\L}^p=\{\u\in\L^p(\mathcal{O}):\nabla\cdot\u=0, \u\cdot\mathbf{n}\big|_{\partial\mathcal{O}}=0\},$ with the norm $\|\u\|_{\widetilde{\L}^p}^p=\int_{\mathcal{O}}|\u(x)|^p\d x$, respectively.
	Let $(\cdot,\cdot)$ denote the inner product in the Hilbert space $\H$ and $\langle \cdot,\cdot\rangle $ represent the induced duality between the spaces $\V$  and its dual $\V'$ as well as $\widetilde{\L}^p$ and its dual $\widetilde{\L}^{p'}$, where $\frac{1}{p}+\frac{1}{p'}=1$. Note that $\H$ can be identified with its dual $\H'$. Moreover, we have the Gelfand triple $\V\hookrightarrow\H \cong\H'\hookrightarrow\V'$ with dense and continuous embedding, and the embedding $\V\hookrightarrow\H$ is compact.
	\subsection{Linear operator}\label{Operator}
	Let $\mathrm{P}_{\H} : \L^p(\mathcal{O}) \to\H$ denote the \emph{Helmholtz-Hodge  projection} (\cite{DFHM}). For $p=2$, $\mathrm{P}_{\H}$ becomes an orthogonal projection and for $2<p<\infty$, it is a bounded linear operator. Let us define
	\begin{equation*}
	\left\{
	\begin{aligned}
	\A\u:&=-\mathrm{P}_{\H}\Delta\u,\;\u\in\D(\A),\\ \D(\A):&=\V\cap\H^{2}(\mathcal{O}).
	\end{aligned}
	\right.
	\end{equation*}
	It should be noted  that the operator $\A$ is a non-negative self-adjoint operator in $\H$ with $\V=\D(\A^{1/2})$ and \begin{align}\label{2.7a}\langle \A\u,\u\rangle =\|\u\|_{\V}^2,\ \textrm{ for all }\ \u\in\V, \ \text{ so that }\ \|\A\u\|_{\V'}\leq \|\u\|_{\V}.\end{align} 
	For the bounded domain $\mathcal{O}$, the operator $\A$ is invertible and its inverse $\A^{-1}$ is bounded, self-adjoint and compact in $\H$. Thus, using spectral theorem, the spectrum of $\A$ consists of an infinite sequence $0< \lambda_1\leq \lambda_2\leq\ldots\leq \lambda_k\leq \ldots,$ with $\lambda_k\to\infty$ as $k\to\infty$ of eigenvalues. Moreover, there exists an orthogonal basis $\{e_k\}_{k=1}^{\infty} $ of $\H$ consisting of eigenvectors of $\A$ such that $\A e_k =\lambda_ke_k$,  for all $ k\in\mathbb{N}$.  We know that any $\u\in\H$ can be expressed as $\u=\sum_{k=1}^{\infty}(\u,e_k)e_k$ and hence $\A\u=\sum_{k=1}^{\infty}\lambda_k(\u,e_k)e_k$, for all $\u\in\D(\A)$. Thus, we deuce that
	\begin{align}\label{poin}
	\|\nabla\u\|_{\mathbb{H}}^2=\langle \A\u,\u\rangle =\sum_{k=1}^{\infty}\lambda_k|(\u,e_k)|^2\geq \lambda_1\sum_{k=1}^{\infty}|(\u,e_k)|^2=\lambda_1\|\u\|_{\mathbb{H}}^2,
	\end{align}
	for all $\u\in\V$ and  
	\begin{align}\label{poin_1}
	\|\A\u\|^2_{\H}=(\A\u, \A\u)&=\sum_{k=1}^{\infty}\lambda_k^2|( \u,e_k)|^2\geq\lambda_1\sum_{k=1}^{\infty}\lambda_k|(\u,e_k)|^2 =\lambda_1\|\nabla\u\|_{\mathbb{H}}^2,
	\end{align}
	for all $\u\in\D(\A)$. 
	\subsection{Bilinear operator}
	Let us define the \emph{trilinear form} $b(\cdot,\cdot,\cdot):\V\times\V\times\V\to\R$ by $$b(\u,\v,\w)=\int_{\mathcal{O}}(\u(x)\cdot\nabla)\v(x)\cdot\w(x)\d x=\sum_{i,j=1}^2\int_{\mathcal{O}}\u_i(x)\frac{\partial \v_j(x)}{\partial x_i}\w_j(x)\d x.$$ If $\u, \v$ are such that the linear map $b(\u, \v, \cdot) $ is continuous on $\V$, the corresponding element of $\V'$ is denoted by $\B(\u, \v)$. We also denote  $\B(\u) = \B(\u, \u)=\mathrm{P}_{\H}(\u\cdot\nabla)\u$.
	An integration by parts yields 
	\begin{equation}\label{b0}
	\left\{
	\begin{aligned}
	b(\u,\v,\v) &= 0,\text{ for all }\u,\v \in\V,\\
	b(\u,\v,\w) &=  -b(\u,\w,\v),\text{ for all }\u,\v,\w\in \V.
	\end{aligned}
	\right.\end{equation}
	The following well-known inequality is due to Ladyzhenskaya (Lemma 1, Chapter I, \cite{OAL}):
	\begin{align}\label{lady}
	\|\v\|_{\L^{4}(\mathcal{O}) } \leq 2^{1/4} \|\v\|^{1/2}_{\L^{2}(\mathcal{O}) } \|\nabla \v\|^{1/2}_{\L^{2}(\mathcal{O}) }, \ \ \ \v\in \H^{1,2}_{0} (\mathcal{O}).
	\end{align}
	Furthermore, an application of the Gagliardo-Nirenberg inequality (Theorem 2.2, \cite{Nirenberg}) yields  the following generalization of \eqref{lady}:
	\begin{align}\label{Gen_lady}
	\|\v\|_{\L^{p}(\mathcal{O}) } \leq C \|\v\|^{1 - \frac{2}{p}}_{\L^{2}(\mathcal{O}) } \|\nabla \v\|^{\frac{2}{p}}_{\L^{2}(\mathcal{O}) }, \ \ \ \v\in \H^{1,2}_{0} (\mathcal{O}),
	\end{align}
	for all $p\in[2,\infty)$. Thus, it is immediate that $\V\subset\wi\L^{r+1}$, for all $r\in[1,\infty)$. Using Ladyzhenskaya's inequality, it is immediate that $\B$ maps $\wi\L^4$ (and so $\V$) into $\V'$ and
	\begin{align*}
	\left|\left<\B(\u,\u),\v\right>\right|=\left|b(\u,\v,\u)\right|\leq \|\u\|_{\L^4}^2\|\nabla\v\|_{\H}\leq \sqrt{2}\|\u\|_{\H}\|\nabla\u\|_{\H}\|\v\|_{\V},
	\end{align*}
	for all $\v\in\V$, so that 
	\begin{align}\label{2.9a}
	\|\B(\u)\|_{\V'}\leq \sqrt{2}\|\u\|_{\H}\|\nabla\u\|_{\H}\leq \frac{\sqrt{2}}{\lambda_1^{1/4}}\|\u\|_{\V}^2,\ \text{ for all }\ \u\in\V,
	\end{align}
	using the Poincar\'e inequality.
	Also, we need the following estimate on the trilinear form $b$ in the sequel (see Chapter 2, section 2.3 \cite{Temam1}),
	\begin{align}
	|b(\u,\v,\w)|\leq C\|\u\|^{1/2}_{\H}\|\u\|^{1/2}_{\V}\|\v\|^{1/2}_{\V}\|\A\v\|^{1/2}_{\H}\|\w\|_{\H},\ \text{ for all }\ \u\in \V, \v\in \D(\A), \w\in \H.\label{b1}
	\end{align}
	\subsection{Nonlinear operator}
	Let us now consider the operator $\mathcal{C}(\u):=\P_{\H}(|\u|^{r-1}\u)$. It is immediate that $\langle\mathcal{C}(\u),\u\rangle =\|\u\|_{\widetilde{\L}^{r+1}}^{r+1}$ and the map $\mathcal{C}(\cdot):\widetilde{\L}^{r+1}\to\widetilde{\L}^{\frac{r+1}{r}}$. Note that $\mathcal{C}'(\u)\v=r\mathrm{P}_{\H}(|\u|^{r-1}\v)$, for all $\u,\v\in\wi\L^{r+1}$, where $\mathcal{C}'(\cdot)$ denotes the Gateaux derivative of $\mathcal{C}(\cdot)$.  Also, for any $r\in [1, \infty)$ and $\u_1, \u_2 \in \V$, we have (see subsection 2.4, \cite{MTM}),
	\begin{align}\label{MO_c}
	\langle\mathcal{C}(\u_1)-\mathcal{C}(\u_2),\u_1-\u_2\rangle \geq 0.
	\end{align}
	
	\subsection{Notations and preliminaries} 
	In this subsection, we  introduce the basic notions and preliminaries on random dynamical systems. Let $(\Omega,\mathscr{F},\{\mathscr{F}_t\}_{t\geq 0},\mathbb{P})$ be a given filtered probability space. 
	
	\begin{definition}
		Suppose that $\mathrm{X}$ is a Polish space, that is, a metrizable complete separable topological space, $\mathcal{B}$ is its Borel $\sigma$-field and $\Im:=(\Omega, \mathscr{F}, \mathbb{P}, \theta)$ is a metric DS. A map $\varphi:\mathbb{R}^{+}\times\Omega\times \mathrm{X}\ni (t, \omega, x)\mapsto \varphi(t, \omega)x \in \mathrm{X} $ is called a \emph{measurable random dynamical system} (RDS) (on $\mathrm{X}$ over $\Im$), if and only if 
		\begin{itemize}
			\item[(i)] $\varphi$ is $(\mathcal{B}(\mathbb{R}^{+}) \otimes \mathscr{F} \otimes \mathcal{B}, \mathcal{B})$-measurable; 
			\item[(ii)] $\varphi$ is a $\theta$-cocycle, that is, $$ \varphi(t + s, \omega, x) = \varphi\big(t, \theta_s \omega, \varphi(s, \omega, x)\big);$$
			\item[(iii)] $\varphi(t, \omega):\mathrm{X}\to \mathrm{X}$ is continuous;
		\end{itemize} 
		The map $\varphi$ is said to be \emph{continuous} if and only if for all $(t, \omega) \in \R^+ \times \Omega, \ \varphi(t, \omega, \cdot):\mathrm{X} \to \mathrm{X}$ is continuous. 
	\end{definition}
	Now we recall the notion of an absorbing random set from the works \cite{BCF,Crauel}, etc. 	Let $\mathfrak{D}$ be the class of closed and bounded random sets on $\mathrm{X}$.
	
	\begin{definition}
		A random set $\mathcal{A}(\omega)$ is said to \emph{absorb} another random set $B(\omega)$ if and only if for all $\omega \in \Omega$, there exists a time $t_B(\omega)\geq 0$ such that  $$ \varphi(t, \theta_{-t}\omega, B(\theta_{-t}\omega))\subset \mathcal{A}(\omega), \  \text{ for all } \ t\geq t_B(\omega). $$ 
		
		The smallest time $t_B(\omega)\geq 0 $ for which above inclusion holds is called the \emph{absorbtion time} of $B(\omega)$ by $\mathcal{A}(\omega).$ 
		
		A random set $\mathcal{A}(\omega)$ is called $\mathfrak{D}$-absorbing if and only if $\mathcal{A}(\omega)$ absorbs every $D(\omega)\in \mathfrak{D}$.
	\end{definition}
	\begin{definition}\label{RA}
		A random set $\mathcal{A}(\omega)$ is a \emph{random $\mathfrak{D}$-attractor} if and only if
		\begin{itemize}
			\item[(i)]$\mathcal{A}$ is a compact random set,
			\item[(ii)]$\mathcal{A}$ is $\varphi$-invariant, that is, $\mathbb{P}$-a.s.,
			
			\begin{center}
				$ \varphi(t,\omega)\mathcal{A}(\omega) = \mathcal{A}(\theta_t\omega),$
			\end{center}
			
			\item[(iii)]$\mathcal{A}$ is $\mathfrak{D}$-attracting, in the sense that, for all $D(\omega)\in \mathfrak{D}$ it holds that $$ \lim_{t\to\infty} d\big(\varphi(t,\theta_{-t}\omega)D(\theta_{-t}\omega),\mathcal{A}(\omega)\big)=0,$$	where $d$ is the Hausdorff semidistance.
		\end{itemize}
	\end{definition}
	\begin{definition}\label{Skew}\emph{(Remark 1.1.8, \cite{Arnold})}
		Given an RDS $\varphi$. Then the mapping $$(\omega,x)\mapsto\big(\theta_t(\omega), \varphi(t,\omega)x\big) =: \Theta_t(\omega,x), \quad\quad t\in \R^+,$$ is a measurable DS on $(\Omega\times X, \mathscr{F}\otimes\mathcal{B})$ which is called the skew product of metric DS $(\Omega, \mathscr{F}, \mathbb{P}, (\theta_t)_{t\in\R})$ and the cocycle $\varphi(t, \omega)$ on $\mathrm{X}$. Conversely, every such measurable skew product DS $\Theta$ defines a cocycle $\varphi$ on its x component, thus a measurable RDS.
	\end{definition}
	\begin{definition}\emph{(See \cite{BL})}
		Let $\varphi$ be given RDS over a metric DS $(\Omega, \mathscr{F}, \mathbb{P}, (\theta_t)_{t\in\R})$. A probability measure $\upeta$ on $(\Omega\times X, \mathscr{F}\otimes\mathcal{B})$ is called an invariant mearure for $\varphi$ if and only if
		\begin{itemize}
			\item [(i)] $\Theta_t$ preserves $\upeta$ (that is, $\Theta_t(\upeta)=\upeta$) for each $t\in \R^+$;
			\item [(ii)] the first marginal of $\upeta$ is $\mathbb{P}$, that is, $\pi_{\Omega}(\upeta)=\mathbb{P},$ where $\pi_{\Omega}:\Omega\times\mathrm{X}\ni (\omega,x)\mapsto\omega\in \Omega.$
		\end{itemize}	
	\end{definition}
	\begin{definition}[\cite{KL}]\label{flat}
		A RDS $\theta$-cocycle $\varphi$ on a Banach space $\mathrm{X}$ is said to be \emph{pullback flattening} if for every random $\mathfrak{D}$-bounded set $\mathcal{B}=\{B(\omega), \omega\in \Omega\}$ in $\mathrm{X}$, for $\delta>0$ and $\omega\in \Omega$, there exists a $T_0(\mathcal{B},\delta,\omega)>0$ and a finite-dimensional subspace $\mathrm{X}_\delta$ of $\mathrm{X}$ such that
		\begin{itemize}
			\item [(i)] $\bigcup_{t\geq T_0} \mathrm{P}_\delta\varphi(t, \theta_{-t}, B(\theta_{-t}\omega))$ is bounded, and
			\item [(ii)] $\|(I-\mathrm{P}_\delta)\big(\bigcup_{t\geq T_0} \varphi(t, \theta_{-t}, B(\theta_{-t}\omega))\big)\|_{\mathrm{X}}<\delta$,
		\end{itemize}
		where $\mathrm{P}_\delta:\mathrm{X}\to \mathrm{X}_\delta$ is a bounded projection and \emph{(ii)} is understood in the sense that $\|(I-\mathrm{P}_\delta)\varphi(t, \theta_{-t}\omega, x_0)\|_{\mathrm{X}}<\delta$, for all $x_0\in B(\theta_{-t}\omega)$ and $t\geq T_0.$ 
	\end{definition}
	\begin{theorem}[\cite{KL}]\label{flat_T}
		Suppose that an RDS $\theta$-cocycle $\varphi$ is pullback flattening and has a random bounded absorbing set. Then it has a unique random attractor. 
	\end{theorem}
	\subsection{Stochastic convective Brinkman-Forchheimer equations}
	In this subsection, we provide the abstract formulation of CBF equations \eqref{1} and discuss about its stochastic counter part. We  also discuss about the existence and uniqeness of strong solutions (in the probabilistic sense) of the stochastic CBF equations. We take the external forcing $\f$ appearing in \eqref{1} independent of time.
	\subsubsection{Abstract formulation }
	On taking orthogonal projection $\P_{\H}$ onto the first equation in \eqref{1}, we obtain 
	\begin{equation}\label{D-CBF}
	\left\{
	\begin{aligned}
	\frac{\d\u}{\d t}+\mu \A\u+\B(\u)+\beta \mathcal{C}(\u)&=\f , \ \ \ t\geq 0, \\ 
	\u(0)&=\u_0,
	\end{aligned}
	\right.
	\end{equation}
	where $\u_0\in\H$ and $\f\in\V'.$ A small random perturbation of the abstract deterministic 2D CBF equations \eqref{D-CBF} is given by
	\begin{equation}\label{S-CBF}
	\left\{
	\begin{aligned}
	\d\u_\varepsilon+[\mu \A\u_\varepsilon+\B(\u_\varepsilon)+\beta \mathcal{C}(\u_\varepsilon)]\d t&=\f \d t +\varepsilon \d\text{W}(t), \ \ \ t\geq 0, \\ 
	\u_\varepsilon(0)&=\u_0,
	\end{aligned}
	\right.
	\end{equation}
	for $r\geq 1$ and $\varepsilon\in(0, 1]$, where we assume that $\u_0\in \H,\ \f\in \V'$ and $\text W(t), \ t\in \R,$ is a two-sided cylindrical Wiener process in $\H$ with its Reproducing Kernal Hilbert Sapce (RKHS) $\mathrm{K}$ defined on some filtered probability space $(\Omega, \mathscr{F}, \{\mathscr{F}_t\}_{t\in \R}, \mathbb{P})$. Remember that  RKHS of a centered Gaussion measure $\nu$ on a separable Banach space $\mathrm{X}$ is a unique Hilbert space $(\mathrm{K}, \|\cdot\|_{\mathrm{K}})$ such that $\mathrm{K}\hookrightarrow \mathrm{X}$ continuously and for each $\varPsi\in \mathrm{X}^{*}$, the random variable $\varPsi$ on probability space $(\mathrm{X}, \nu)$ is normal with mean $0$ and variance $\|\varPsi\|^{2}_{\mathrm{K}}$ (\cite{DZ1}). 
	
	In this paper, we assume that RKHS $\mathrm{K}$ satisfies the following assumption:
	\begin{assumption}\label{assump}
		$ \mathrm{K} \subset \V \cap \H^{2} (\mathcal{O})$ is a Hilbert space such that for some $\delta\in (0, 1/2),$
		\begin{align}\label{A1}
		\A^{-\delta} : \mathrm{K} \to \V \cap \H^{2} (\mathcal{O}) \   \text{ is }\ \gamma \text{-radonifying.}
		\end{align}
	\end{assumption}

	\begin{remark}
		Since $\D(\A)=\V \cap \H^{2} (\mathcal{O})$, Assumption \ref{assump} can be reformulated in the following way also (see \cite{BL1}). $\mathrm{K}$ is a Hilbert space such that $\mathrm{K}\subset\D(\A)$ and for some $\delta\in(0,1/2)$, the map \begin{align}\label{34}
		\A^{-\delta-1} : \mathrm{K} \to \H \   \text{ is }\ \gamma \text{-radonifying.}
		\end{align}
		Note that \eqref{34} also says that the mapping $\A^{-\delta-1} : \mathrm{K} \to \H$  is Hilbert-Schmidt.	Since $\mathcal{O}$ is a bounded domain, then $\A^{-s}:\H\to\H$ is Hilbert-Schmidt if and only if $\sum_{j=1}^{\infty} \lambda_j^{-2s}<\infty,$ where  $\A e_j=\lambda_j e_j, j\in \N$ and $e_j$ is an orthogonal basis of $\H$. In bounded domains, we know that $\lambda_j\sim j$ and hence $\A^{-s}$ is Hilbert-Schmidt if and only if $s>\frac{1}{2}.$ In other words, with $\mathrm{K}=\D(\A^{s+1}),$ the embedding $\mathrm{K}\hookrightarrow\V\cap\H^2(\mathcal{O})$ is $\gamma$-radonifying if and only if $s>\frac{1}{2}.$ Thus, Assumption \ref{assump} is satisfied for any $\delta>0.$ In fact, the condition \eqref{A1} holds if and only if the operator $\A^{-(s+1+\delta)}:\H\to \V\cap\H^2(\mathcal{O})$ is $\gamma$-radonifying. The requirement of  $\delta<\frac{1}{2}$ in Assumption \ref{assump} is necessary because we need (see subsection \ref{O_up}) the corresponding Ornstein-Uhlenbeck process has to take values in $\V\cap\H^2(\mathcal{O})$.
	\end{remark}
	\subsection{Solution to 2D SCBF equations }
	In this subsection, we provide the definition of a pathwise unique strong  solution in the probabilistic sense to the system \eqref{S-CBF}.
	\begin{definition}\label{def3.2}
		Let $\u_0 \in \H$, $r\geq1$, $\f\in \V'$ and $\mathrm{W}(t), t\in \R$ is two sided Wiener process in $\H$ with its RKHS $\mathrm{K}$. An $\H$-valued $\{\mathscr{F}_t\}_{t\geq 0}$-adapted stochastic process $\u_\varepsilon(t), \ t\geq 0,$ is called a strong solution to the system \eqref{S-CBF} if the following conditios are satisfied:
		\begin{itemize}
			\item [(i)] the process $\u_\varepsilon\in \mathrm{L}^{2}(\Omega; \mathrm{L}^{\infty}(0, T; \H) \cap \mathrm{L}^{2}(0, T; \V))\cap \mathrm{L}^{r+1}(\Omega; \mathrm{L}^{r+1}(0,T;\widetilde{\L}^{r+1}(\mathcal{O})))$  with $\mathbb{P}$-a.s., trajectories in $\mathrm{C}([0, T]; \H) \cap \mathrm{L}^{2}(0, T; \V).$ 
			\item [(ii)] the following equality holds for every $t\in[0,T]$ and for any $\psi\in \V$, $\mathbb{P}$-a.s.  
			\begin{align}\label{W-SCBF}
			(\u_\varepsilon(t), \psi) &= (\u_0, \psi) - \int_{0}^{t} \langle \mu \A\u_\varepsilon(s)+\B(\u_\varepsilon(s))+\beta \mathcal{C}(\u_\varepsilon(s)) , \psi\rangle \d s  +  \int_{0}^{t}\langle \f , \psi \rangle \d s\nonumber \\&\quad +  \int_{0}^{t} (\d\mathrm{W}(s), \psi).
			\end{align}
		\end{itemize}
	\end{definition} 
	\begin{theorem}[\cite{MTM}]
		Let $\u_0 \in \H$, $\f\in \V'$ and Assumption \ref{assump} be satisfied. Then, for $r\geq 1$, there exists a unique strong solution $\u_\varepsilon(\cdot)$ to the system \eqref{S-CBF} in the sense of Definition \ref{def3.2}. In addition, let $\u_0\in \V$ and $\f\in\H$. Then, for $r\geq 1,$ the pathwise unique strong solution $\u_\varepsilon(\cdot)$ to the system \eqref{S-CBF} satisfies the following regularity:
		$$\u_\varepsilon\in \mathrm{L}^{2}(\Omega; \mathrm{L}^{\infty}(0, T; \V) \cap \mathrm{L}^{2}(0, T; \D(\A)))\cap \mathrm{L}^{r+1}(\Omega; \mathrm{L}^{r+1}(0,T;\widetilde{\L}^{3(r+1)}(\mathcal{O}))).$$ Moreover, the $\mathscr{F}_t$-adapted paths of $\u_\varepsilon(\cdot)$ are continuous with trajectories in $\mathrm{C}([0,T];\V), \mathbb{P}$-a.s.
	\end{theorem}
	\section{RDS generated by the 2D SCBF equations}\label{sec3}\setcounter{equation}{0}
	In this section, we construct the metric dynamical system and the  random dynamical system for the model \eqref{S-CBF}. 
	\subsection{Wiener process}
	
	Let us denote $\mathrm{X} = \V \cap \H^{2} (\mathcal{O})$ and let $\mathrm{E}$ denote the completion of $\A^{-\delta}\mathrm{X}$ with respect to the image norm $\|x\|_{\mathrm{E}}  = \|\A^{-\delta} x\|_{\mathrm{X}} , \ \text{for } \ x\in \mathrm{X}, \text{where } \|\cdot\|_{\mathrm{X}} = \|\cdot\|_{\V} + \|\cdot\|_{\H^{2} }.$ Note that $\mathrm{E}$ is a separable Banach space (see \cite{Brze2}).
	
	For $\xi \in(0, 1/2),$ we set 
	$$ \|\omega\|_{\C^{\xi}_{1/2} (\mathbb{R},\mathrm{E})} = \sup_{t\neq s \in \mathbb{R}} \frac{\|\omega(t) - \omega(s)\|_{\mathrm{E}}}{|t-s|^{\xi}(1+|t|+|s|)^{1/2}} .$$
	We also define
	\begin{align*}
	\C^{\xi}_{1/2} (\mathbb{R}, \mathrm{E}) &= \left\{ \omega \in \C(\mathbb{R}, \mathrm{E}) : \omega(0)=0,\  \|\omega\|_{\C^{\xi}_{1/2} (\mathbb{R},\mathrm{E})} < \infty \right\},\\
	\Omega(\xi, \mathrm{E})&=\text{the closure of } \{ \omega \in \C^\infty_0 (\mathbb{R}, \mathrm{E}) : \omega(0) = 0 \} \ \text{ in } \ \C^{\xi}_{1/2} (\mathbb{R},\mathrm{E}).
	\end{align*}
	The space $\Omega(\xi, \mathrm{E})$ is a separable Banach space. 
	Let us denote $\mathscr{F}$ for the Borel $\sigma$-algebra on $\Omega(\xi, \mathrm{E}).$ For $\xi\in (0, 1/2)$, there exists a Borel probability measure $\mathbb{P}$ on $\Omega(\xi, \mathrm{E})$ (see \cite{Brze}) such that the canonical process $w_t, \ t\in \mathbb{R},$ defined by 
	\begin{align}\label{Wp}
	w_t(\omega) := \omega(t), \ \ \omega \in \Omega(\xi, \mathrm{E}),
	\end{align}
	is an $\mathrm{E}$-valued two-sided Wiener process. 
	
	For $t\in \mathbb{R},$ let $\mathscr{F}_t := \sigma \{ w_s : s \leq t \}.$ Then there exists a bounded linear map $\text{W}_t : \mathrm{K} \to \mathrm{L}^2(\Omega(\xi, \mathrm{E}), \mathscr{F}_t  ,  \mathbb{P}).$ Moreover, the family $\{\text{W}_t\}_{t\in \mathbb{R}}$ is a $\mathrm{K}$-cylindrical Wiener process on a filtered probability space $(\Omega(\xi, \mathrm{E}), \mathscr{F}, (\mathscr{F}_t)_{t \in \mathbb{R}} , \mathbb{P})$ (see also \cite{BP}).
	
	On the space $\Omega(\xi, \mathrm{E}),$ we consider a flow $\theta = \{\theta_t\}_{t\in \mathbb{R}}$ defined by
	$$ \theta_t \omega(\cdot) = \omega(\cdot + t) - \omega(t), \ \ \ \omega\in \Omega(\xi, \mathrm{E}), \ \ t\in \mathbb{R}.$$ 
	This flow keeps the space $\Omega(\xi, \mathrm{E})$ invariant. It is obvious that for each $ t \in \R,\ \theta_t$ preserves $\mathbb{P}.$ 

	\subsection{Ornstein-Uhlenbeck process}\label{O_up}

	In this subsection, we define an Ornstein-Uhlenbeck processes under Assumption \ref{assump} (for more details see section 4,  \cite{KM}). For $\delta$ as in Assumption \ref{assump}, $\mu > 0, \ \alpha \geq 0, \ \xi \in (\delta, 1/2)$ and $ \omega \in C^{\xi}_{1/2} (\mathbb{R}, \mathrm{E})$ (so that $(\mu \A + \alpha I)^{-\delta}\omega \in \C^{\xi}_{1/2} (\mathbb{R}, \mathrm{X})$), we define 
	\begin{align}\label{DOu}
	\z_{\alpha}(\omega)(t)&:=\int_{-\infty}^{t} (\mu \A + \alpha I)^{1+\delta} e^{-(t-r)(\mu \A + \alpha I)} [(\mu \A + \alpha I)^{-\delta}\omega(t) - (\mu \A + \alpha I)^{-\delta}\omega(r)]\d r, 
	\end{align}
	for any $t\geq 0.$ Hence $\z_{\alpha}(t)$ is the solution of the following equation:
	\begin{align}\label{OuE}
	\frac{\d\z_{\alpha} (t)}{\d t} + (\mu \A + \alpha I)\z_{\alpha} = \dot{\omega} (t), \ \ t\in \mathbb{R}.
	\end{align}
	Analogously to our definition \eqref{Wp} of the Wiener process $w(t), \ t\in \R,$ we can view the formula \eqref{DOu} as a definition of a process $\z_{\alpha}(t), \ t\in \R,$ on the probability space $(\Omega(\xi, \mathrm{E}), \mathscr{F}, \mathbb{P}).$ Equation \eqref{OuE} suggests that this process is an Ornstein-Uhlenbeck process. In fact we have the following result.
	\begin{proposition}[Proposition 6.10, \cite{BL}]\label{SOUP}
		The process $\z_{\alpha}(t), \ t\in \mathbb{R},$ is stationary Ornstein-Uhlenbeck process on $(\Omega(\xi, \mathrm{E}), \mathscr{F}, \mathbb{P})$ . It is a solution of the equation 
		\begin{align}\label{OUPe}
		\d\z_{\alpha}(t) + (\mu \A + \alpha I)\z_{\alpha}(t) \d t = \d\mathrm{W}(t), \ \ t\in \mathbb{R},
		\end{align}
		that is, for all $t\in \mathbb{R},$ 
		\begin{align}\label{oup}
		\z_\alpha (t) = \int_{-\infty}^{t} e^{-(t-s)(\mu \A + \alpha I)} \d\mathrm{W}(s),
		\end{align}
		$\mathbb{P}$-a.s.,	where the integral is the It\^o integral on the M-type 2 Banach space $\mathrm{X}$ in the sense of \cite{Brze1}. 
		In particular, for some $C$ depending on $\mathrm{X}$,
		\begin{align}\label{E-OUP}
		\mathbb{E}\left[\|\z_{\alpha} (t)\|^2_{\mathrm{X}} \right]&= \mathbb{E}\left[\left\|\int_{-\infty}^{t} e^{-(t-s)(\mu \A + \alpha I)} \d\mathrm{W}(s)\right\|^2_{\mathrm{X}} \right]\leq C\int_{-\infty}^{t} \|e^{-(t-s)(\mu \A +  \alpha I)}\|^2_{\gamma(\mathrm{K},\mathrm{X})} \d s \nonumber\\&=C \int_{0}^{\infty}  e^{-2\alpha s} \|e^{-\mu s \A}\|^2_{\gamma(\mathrm{K},\mathrm{X})} \d s.
		\end{align} 
		Moreover, $\mathbb{E}\left[\|\z_{\alpha} (t)\|^2_{\mathrm{X}}\right]$ tends to $0$ as $\alpha \to \infty.$
		
	\end{proposition}
	\begin{remark}\label{stationary}
		By Proposition 4.1 \cite{KM}, we write the following result for the Ornstein-Uhlenbeck process given in Proposition \ref{SOUP}:
		\begin{align}
		\z_\alpha(\theta_s \omega)(t) = \z_\alpha(\omega)(t+s), \ \ t, s \in \mathbb{R},
		\end{align}and  
		\begin{align}\label{O-U_conti}
		\z_\alpha\in\mathrm{L}^{q} (a, b; \mathrm{X})
		\end{align}
		where $q\in [1, \infty].$
	\end{remark}
	Since by Proposition \ref{SOUP}, the process $\z_{\alpha}(t), \ t\in \R $ is an $\mathrm{X}$-valued stationary and ergodic. Hence, by the Strong law of Large Numbers, we have  (see \cite{DZ} for a similar argument)
	\begin{align}\label{SLLN}
	\lim_{t \to \infty} \frac{1}{t} \int_{-t}^{0} \|\z_{\alpha}(s)\|^{2}_{\mathrm{X}} \d s = \mathbb{E}\left[ \|\z_{\alpha}(0)\|^{2}_{\mathrm{X}}\right], \ \  \mathbb{P}\text{-a.s., on }\ \C^{\xi}_{1/2}(\R, \mathrm{X}).
	\end{align}
	Therefore it follows from Proposition \ref{SOUP} that we can find $\alpha_0$ such that for all $\alpha \geq \alpha_0,$ 
	\begin{align}\label{Bddns}
	\mathbb{E}\left[\|\z_{\alpha} (0)\|^{2}_{\mathrm{X}}\right] \leq \frac{\mu^2\lambda_1}{16},
	\end{align}  
	where $\lambda_1$ is the constant appearing in inequality \eqref{poin} (Poincar\'e inequality).
	
	By $\Omega_{\alpha}(\xi, \mathrm{E})$, we denote the set of those $\omega\in \Omega(\xi, \mathrm{E})$ for which the equality \eqref{SLLN} holds true. Therefore, we fix $\xi \in (\delta, 1/2)$ and set $$\Omega := \hat{\Omega}(\xi, \mathrm{E}) = \bigcap^{\infty}_{n=0} \Omega_{n}(\xi, \mathrm{E}).$$
	
	For reasons that will become clear later, we take as a model of a metric DS the quadruple $(\Omega, \hat{\mathscr{F}}, \hat{\mathbb{P}}, \hat{\theta}),$ where $\hat{\mathscr{F}}$, $\hat{\mathbb{P}}$, $\hat{\theta}$ are respectively the natural restrictions of $\mathscr{F}$, $\mathbb{P}$ and $\theta$ to $\Omega.$	
	
	\begin{proposition}\label{m-DS}
		The quadruple $(\Omega, \hat{\mathscr{F}}, \hat{\mathbb{P}}, \hat{\theta})$ is a metric DS. For each $\omega\in \Omega,$ the limit in \eqref{SLLN} exists.
	\end{proposition}

	Let us now formulate an immediate and important consequence of the above result.
	
	\begin{corollary}\label{Bddns1}
		For each $\omega\in \Omega,$ there exists $t_0=t_0 (\omega) \geq 0 $ such that 
		\begin{align}
		\frac{8}{\mu} \int_{-t}^{0} \|\z_{\alpha}(s)\|^{2}_{\mathrm{X}} \d s \leq \frac{\mu \lambda_1 t}{2}, \ \  t\geq t_0.
		\end{align}
		Also, since the embedding $\mathrm{X}\hookrightarrow \V$ is a contraction, we have
		\begin{align}\label{Bddns2}
		\frac{8}{\mu} \int_{-t}^{0} \|\z_{\alpha}(s)\|^{2}_{\V} \d s \leq \frac{\mu \lambda_1 t}{2}, \ \  t\geq t_0.
		\end{align}
	\end{corollary}
	\subsection{Random dynamical system}
	Let us recall that Assumption \ref{assump} is satisfied and that $\delta$ has the property stated there. We take a fixed $\mu > 0$ and some parameter $\alpha\geq 0$. We also fix $\xi \in (\delta, 1/2)$.
	
	Denote by $\v^{\alpha}_\varepsilon(t)=\u_\varepsilon(t) -\varepsilon \z_{\alpha}(\omega)(t)$, then $\v^{\alpha}_\varepsilon(t)$ satisfies the following abstract random dynamical system:
	\begin{equation}\label{cscbf}
	\left\{
	\begin{aligned}
	\frac{\d\v_\varepsilon(t)}{\d t} &= -\mu \A\v_\varepsilon (t)- \B(\v_\varepsilon(t) + \varepsilon\z(t)) - \beta \mathcal{C}(\v_\varepsilon (t)+\varepsilon \z(t)) +\varepsilon \alpha \z(t) + \f, \\
	\v_\varepsilon(0)&= \u_0 - \varepsilon\z_{\alpha}(0).
	\end{aligned}
	\right.
	\end{equation}
	Because $\z_{\alpha}(\omega) \in \C_{1/2} (\mathbb{R}, \mathrm{X}), \ \z_{\alpha}(\omega)(0)$ is a well defined element of $\V$. In what follows, we provide the definition of weak as well as strong solutions for the system \eqref{cscbf}.
		\begin{definition}
		Assume that $\u_0\in \H$ and  $\f\in \V'$. Let $T>0$ be any fixed time, a function $\v_\varepsilon(\cdot)$ is called a \emph{weak solution} of the problem \eqref{cscbf} on time interval $[0, T]$, if $$\v_\varepsilon\in \mathrm{C}([0,T];\H)\cap\mathrm{L}^2(0, T;\V), \  \frac{\d\v_\varepsilon}{\d t}\in \mathrm{L}^2(0,T;\V')$$ and it satisfies 
		\begin{itemize}
			\item [(i)] for any $\psi\in \V,$ 
			\begin{align}\label{W-CSCBF}
			\left<\frac{\d\v_\varepsilon(t)}{\d t}, \psi\right>& =  - \left< \mu \A\v_\varepsilon(t)+\B(\v_\varepsilon(t)+\varepsilon\z_{\alpha}(t))+\beta \mathcal{C}(\v_\varepsilon(s)+\varepsilon\z_{\alpha}(t)) , \psi\right> +  \left<\f , \psi \right> \nonumber\\&\quad +   \left(\varepsilon\alpha\z_{\alpha}(t), \psi\right),
			\end{align}
			for all $t\in[0,T]$.
			\item [(ii)] $\v_\varepsilon(t)$ satisfies the following initial data$$\v_\varepsilon(0)=\u_0-\varepsilon\z_{\alpha}(0).$$
		\end{itemize}
	\end{definition}
	\begin{definition}
		Assume that $\u_0\in \V$ and  $\f\in \H$. Let $T>0$ be any fixed time, a function $\v_\varepsilon(\cdot)$ is called a \emph{strong solution} of the problem \eqref{cscbf} on time interval $[0, T]$, if $$\v_\varepsilon\in \mathrm{C}([0,T];\V)\cap\mathrm{L}^2(0, T;\D(\A)), \  \frac{\d\v_\varepsilon}{\d t}\in \mathrm{L}^2(0,T;\H)$$ and it satisfies \eqref{cscbf} as an equality in $\H$ for a.e. $t\in(0,T)$. 
	\end{definition}
	Since, $\u_\varepsilon(\cdot)$ is the unique solution to the problem \eqref{S-CBF} and $\z_{\alpha}(\cdot)$ is the unique solution to the problem \eqref{OUPe}, one can easily obtain a unique solution $\v^{\alpha}_\varepsilon(t)$  to the problem \eqref{cscbf}.
	For the next theorem, we take $\f$ is dependent on $t$.
	
	\begin{theorem}\label{RDS_Conti1}
		Assume that, for some $T >0$ fixed, $\u_0^n \to \u_0$ in $\H$, 
		\begin{align*}
		\z_n \to \z\ \text{ in }\ \mathrm{L}^{\infty} (0, T; \V)\cap \mathrm{L}^2(0, T;\D(\A)),\ \ \f_n \to \f \ \text{ in }\ \mathrm{L}^2 (0, T; \V').
		\end{align*}
		Let us denote by $\v_\varepsilon(t, \z)\u_0$, the solution of the problem \eqref{cscbf} and by $\v_\varepsilon(t, \z_n)\u_0^n$, the solution of the problem \eqref{cscbf} with $\z, \f, \u_0$ being replaced by $\z_n, \f_n, \u_0^n$, respectively. Then $$\v_\varepsilon(\cdot, \z_n)\u_0^n \to \v_\varepsilon(\cdot, \z)\u_0 \ \text{ in } \ \mathrm{C}([0,T];\H)\cap\mathrm{L}^2 (0, T;\V).$$
		In particular, $\v_\varepsilon(T, \z_n)\u_0^n \to \v_\varepsilon(T, \z)\u_0$ in $\H$.
	\end{theorem}	
	\begin{proof}
		See Theorem 4.8, \cite{KM}. 
	\end{proof}
	\begin{theorem}\label{RDS_Conti}
		Assume that, for some $T >0$ fixed, $\u_0^n \to \u_0$ in $\V$, 
		\begin{align*}
		\z_n \to \z\ \text{ in }\ \mathrm{L}^{\infty} (0, T; \V)\cap \mathrm{L}^2(0, T;\D(\A)),\ \ \f_n \to \f \ \text{ in }\ \mathrm{L}^2 (0, T; \H).
		\end{align*}
		Let us denote by $\v_\varepsilon(t, \z)\u_0$, the solution of the problem \eqref{cscbf} and by $\v_\varepsilon(t, \z_n)\u_0^n$, the solution of the problem \eqref{cscbf} with $\z, \f, \u_0$ being replaced by $\z_n, \f_n, \u_0^n$, respectively. Then $$\v_\varepsilon(\cdot, \z_n)\u_0^n \to \v_\varepsilon(\cdot, \z)\u_0 \ \text{ in } \ \mathrm{C}([0,T];\V)\cap\mathrm{L}^2 (0, T;\D(\A)).$$
		In particular, $\v_\varepsilon(T, \z_n)\u_0^n \to \v_\varepsilon(T, \z)\u_0$ in $\V$.
	\end{theorem}	
	\begin{proof}
		In order to simplify the proof, we introduce the following notations:
		$$\v_n (t) = \v_\varepsilon(t, \z_n)\u_0^n, \ \  \v(t) = \v_\varepsilon(t, \z)\u_0,\ \   \y_n (t)= \v_\varepsilon(t, \z_n)\u_0^n - \v_\varepsilon(t, \z)\u_0, \ \ \ t\in[0, T],$$     and 
		$$ \hat{ \z}_n = \z_n - \z, \ \  \hat{ \f}_n = \f_n - \f.$$
		It is easy to see that $\y_n$ solves the following initial value problem:
		\begin{equation}\label{cscbf_n}
		\left\{
		\begin{aligned}
		\frac{\d\y_n(t)}{\d t} &= -\mu \A\y_n (t) - \B(\v_n (t) + \varepsilon\z_n (t)) + \B(\v(t) + \varepsilon\z(t)) - \beta \mathcal{C}(\v_n (t) +\varepsilon \z_n (t)) \\& \ \ \ + \beta \mathcal{C}(\v(t) + \varepsilon\z(t)) +\varepsilon \alpha \hat{ \z}_n (t)+ \hat{ \f}, \\
		\y_n(0)&= \u_0^n - \u_0.
		\end{aligned}
		\right.
		\end{equation}
		Taking the inner product with $\A\y_n(\cdot)$ in the first equation in \eqref{cscbf_n}, and then using  \eqref{2.7a} and \eqref{b0}, we get 
		\begin{align}\label{Energy_esti_n_1}
		&	\frac{1}{2} \frac{\d}{\d t}\|\y_n (t)\|^2_{\V} \nonumber\\ &=-\mu \|\A\y_n(t)\|^2_{\H} -b(\y_n(t),\v_n(t)+\varepsilon\z_n(t),\A\y_n(t))-\varepsilon b(\hat{ \z}_n(t),\v_n(t)+\varepsilon\z_n(t),\A\y_n(t))\nonumber\\&\quad- b(\v(t)+\varepsilon\z(t),\y_n(t),\A\y_n(t))-\varepsilon b(\v(t)+\varepsilon\z(t),\hat{ \z}_n(t),\A\y_n(t))\nonumber\\&\quad-\beta(\mathcal{C}(\v_n(t)+\varepsilon\z_n(t))-\mathcal{C}(\v(t)+\varepsilon\z(t)), \A\y_n(t))+ \varepsilon\alpha \big(\hat{ \z}_n(t), \y_n(t)\big) + (\hat{ \f}(t),\y_n(t)), 
		\end{align}
		for a.e. $t\in[0,T]$. Now, using \eqref{poin}, \eqref{b1}, $0<\varepsilon\leq1$ and Young's inequality, we have
		\begin{align}
		\big|b(\y_n,\v_n+\varepsilon\z_n,\A\y_n)\big|&\leq C\|\y_n\|_{\V}\|\v_n+\varepsilon\z_n\|^{\frac{1}{2}}_{\V}\|\A\v_n+\varepsilon\A\z_n\|^{\frac{1}{2}}_{\H}\|\A\y_n\|_{\H}\nonumber\\&\leq\frac{\mu}{14}\|\A\y_n\|^{2}_{\H}+C\|\y_n\|^2_{\V}\|\v_n+\varepsilon\z_n\|_{\V}\|\A\v_n+\varepsilon\A\z_n\|_{\H},\label{Energy_esti_n_2}\\
		\big|\varepsilon b(\hat{ \z}_n,\v_n+\varepsilon\z_n,\A\y_n)\big|&\leq\big| b(\hat{ \z}_n,\v_n+\varepsilon\z_n,\A\y_n)\big|\nonumber\\&\leq C\|\hat{ \z}_n\|_{\V}\|\v_n+\varepsilon\z_n\|^{\frac{1}{2}}_{\V}\|\A\v_n+\varepsilon\A\z_n\|^{\frac{1}{2}}_{\H}\|\A\y_n\|_{\H}\nonumber\\&\leq\frac{\mu}{14}\|\A\y_n\|^{2}_{\H}+C\|\hat{ \z}_n\|^2_{\V}\|\v_n+\varepsilon\z_n\|_{\V}\|\A\v_n+\varepsilon\A\z_n\|_{\H},\\
		\big|b(\v+\varepsilon\z,\y_n,\A\y_n)\big|&\leq C\|\v+\varepsilon\z\|_{\V}\|\y_n\|^{\frac{1}{2}}_{\V}\|\A\y_n\|^{\frac{3}{2}}_{\H}\leq\frac{\mu}{14}\|\A\y_n\|^{2}_{\H}+C\|\v+\varepsilon\z\|^4_{\V}\|\y_n\|^{2}_{\V},\\
		\big|\varepsilon b(\v+\varepsilon\z,\hat{ \z}_n,\A\y_n)\big|&\leq\big| b(\v+\varepsilon\z,\hat{ \z}_n,\A\y_n)\big|\leq C\|\v+\varepsilon\z\|_{\V}\|\hat{ \z}_n\|^{\frac{1}{2}}_{\V}\|\A\hat{ \z}_n\|^{\frac{1}{2}}_{\H}\|\A\y_n\|_{\H}\nonumber\\ &\leq \frac{\mu}{14}\|\A\y_n\|^{2}_{\H} +C\|\v+\varepsilon\z\|^2_{\V}\|\hat{ \z}_n\|_{\V}\|\A\hat{ \z}_n\|_{\H}.
		\end{align}
		Now by using Taylor's formula and Gagliardo-Nirenberg's inequality, we find 
		\begin{align}
		&\beta\big|(\mathcal{C}(\v_n+\varepsilon\z_n)-\mathcal{C}(\v+\varepsilon\z), \A\y_n)\big|\nonumber\\&= \beta\bigg|\bigg(\int_{0}^{1} \big[\mathcal{C}'\big(\theta(\v_n+\varepsilon\z_n) + (1-\theta)(\v+\varepsilon\z)\big)\big((\v_n+\varepsilon\z_n)-(\v+\varepsilon\z)\big)\big] \d \theta , \A\y_n\bigg)\bigg|\nonumber\\&= \beta\bigg|\bigg( \int_{0}^{1}r\P_{\H}\big[|\big(\theta(\v_n+\varepsilon\z_n) + (1-\theta)(\v+\varepsilon\z)\big)|^{r-1}\big((\v_n+\varepsilon\z_n)-(\v+\varepsilon\z)\big)\big] \d\theta , \A\y_n\bigg)\bigg|\nonumber\\& \leq r\beta 2^{r-2} \|\v_n+\varepsilon\z_n\|^{r-1}_{\wi \L^{2r}} \|(\v_n+\varepsilon\z_n)-(\v+\varepsilon\z)\|_{\wi \L^{2r}} \|\A\y_n\|_{\H}\nonumber\\&\quad+r\beta 2^{r-2} \|\v+\varepsilon\z\|^{r-1}_{\wi \L^{2r}} \|(\v_n+\varepsilon\z_n)-(\v+\varepsilon\z)\|_{\wi \L^{2r}} \|\A\y_n\|_{\H}\nonumber\\& \leq r\beta C2^{r-2} \|\v_n+\varepsilon\z_n\|^{r-1}_{\V} \|\y_n+\varepsilon\hat{ \z}_n\|_{\V} \|\A\y_n\|_{\H}+r\beta C2^{r-2} \|\v+\varepsilon\z\|^{r-1}_{\V} \|\y_n+\varepsilon\hat{ \z}_n\|_{\V} \|\A\y_n\|_{\H}\nonumber\\ &\leq r\beta C2^{r-2} \|\v_n+\varepsilon\z_n\|^{r-1}_{\V} \|\y_n\|_{\V} \|\A\y_n\|_{\H}+ \varepsilon r\beta C2^{r-2} \|\v_n+\varepsilon\z_n\|^{r-1}_{\V} \|\hat{ \z}_n\|_{\V} \|\A\y_n\|_{\H}\nonumber\\ &\quad+ r\beta C2^{r-2} \|\v+\varepsilon\z\|^{r-1}_{\V} \|\y_n\|_{\V} \|\A\y_n\|_{\H} + \varepsilon r\beta C2^{r-2} \|\v+\varepsilon\z\|^{r-1}_{\V} \|\hat{ \z}_n\|_{\V} \|\A\y_n\|_{\H}\nonumber\\&\leq\frac{\mu}{14} \|\A\y_n\|^2_{\H} + C\big[\|\v_n+\varepsilon\z_n\|^{2r-2}_{\V} +\|\v+\varepsilon\z\|^{2r-2}_{\V}\big]\|\y_n\|^2_{\V}\nonumber\\&\quad+ C\big[\|\v_n+\z_n\|^{2r-2}_{\V} +\|\v+\varepsilon\z\|^{2r-2}_{\V}\big]\|\hat{ \z}_n\|^2_{\V}.
		\end{align}
		Using $0<\varepsilon\leq1$, H\"older's and Young's inequalities, we estimate
		\begin{align}
		\varepsilon\alpha\big\langle\hat{ \z}_n, \A\y_n\big\rangle &\leq \alpha \|\A\y_n\|_{\H} \|\hat{ \z}_n\|_{\H} \leq \frac{\mu }{14} \|\A\y_n\|^2_{\H} + C \|\hat{ \z}_n\|^2_{\H} \leq \frac{\mu}{14} \|\A\y_n\|^2_{\H} + C \|\hat{ \z}_n\|^2_{\V},\\
		\big\langle\hat{ \f}_n, \A\y_n\big\rangle &\leq  \|\A\y_n\|_{\H} \|\hat{ \f}_n\|_{\H} \leq \frac{\mu}{14} \|\A\y_n\|^2_{\H} + C \|\hat{ \f}_n\|^2_{\H}.\label{Energy_esti_n_3}
		\end{align}
		Combining \eqref{Energy_esti_n_2}-\eqref{Energy_esti_n_3} and using it in \eqref{Energy_esti_n_1}, we obtain 
		\begin{align*}
		\frac{\d}{\d t}\|\y_n (t)\|^2_{\V} &+ \mu \|\A\y_n(t)\|^2_{\H}\nonumber\\ \leq & C\bigg[\|\v_n(t)+\varepsilon\z_n(t)\|_{\V}\|\A\v_n(t)+\varepsilon\A\z_n(t)\|_{\H}+ \|\v_n(t)+\varepsilon\z_n(t)\|^{2r-2}_{\V} \nonumber\\&+\|\v(t)+\varepsilon\z(t)\|^{2r-2}_{\V}+\|\v(t)+\varepsilon\z(t)\|^4_{\V}\bigg]\|\y_n(t)\|^2_{\V}\nonumber\\&+ C\bigg[\|\v_n(t)+\varepsilon\z_n(t)\|_{\V}\|\A\v_n(t)+\varepsilon\A\z_n(t)\|_{\H} +\|\v(t)+\varepsilon\z(t)\|^{2r-2}_{\V}\nonumber\\&+\|\v_n(t)+\varepsilon\z_n(t)\|^{2r-2}_{\V} +1\bigg]\|\hat{ \z}_n(t)\|^2_{\V}+  C \|\hat{ \f}_n(t)\|^2_{\H}\nonumber\\&+C\|\v(t)+\varepsilon\z(t)\|^2_{\V}\|\hat{ \z}_n(t)\|_{\V}\|\A\hat{ \z}_n(t)\|_{\H}, 
		\end{align*}
		for a.e. $t\in[0,T]$. Now integrating from $0$ to $t$ to the above inequality, we get
		\begin{align}\label{324}
		\|\y_n(t)\|^2_{\V}+\mu\int_{0}^{t} \|\A\y_n(s)\|^2_{\H}\d s\leq \|\y_n(0)\|^2_{\V} +C\int_{0}^{t} \upalpha_n(s)\|\y_n(s)\|^2_{\V}\d s + C \int_{0}^{t} \upbeta_n(s)\d s,
		\end{align}
		for all $t\in [0,T],$ where
		\begin{align*}
		\upalpha_n&=\|\v_n+\varepsilon\z_n\|_{\V}\|\A\v_n+\varepsilon\A\z_n\|_{\H}+ \|\v_n+\varepsilon\z_n\|^{2r-2}_{\V} +\|\v+\varepsilon\z\|^{2r-2}_{\V}+\|\v+\varepsilon\z\|^4_{\V},\nonumber\\
		\upbeta_n&= 	\Big[\|\v_n+\varepsilon\z_n\|_{\V}\|\A\v_n+\varepsilon\A\z_n\|_{\H} +\|\v+\varepsilon\z\|^{2r-2}_{\V}+\|\v_n+\varepsilon\z_n\|^{2r-2}_{\V} +1\Big]\|\hat{ \z}_n\|^2_{\V}\nonumber\\&\quad+   \|\hat{ \f}_n\|^2_{\H}+\|\v+\varepsilon\z\|^2_{\V}\|\hat{ \z}_n\|_{\V}\|\A\hat{ \z}_n\|_{\H}.
		\end{align*}
		Then by an application of Gronwall's inequality, we obtain 
		\begin{align}\label{Energy_esti_n_4}
		\|\y_n(t)\|^2_{\V}&\leq \bigg(\|\y_n(0)\|^2_{\V}+ C \int_{0}^{T} \upbeta_n(s) \d s\bigg) e^{C\int_{0}^{T} \upalpha_n (s) \d s},
		\end{align}
		for all $t\in [0,T].$  On the other hand, we find
		\begin{align*}
		\int_{0}^{T} \upbeta_n(s) \d s&\leq T^{\frac{1}{2}}\|\v_n+\varepsilon\z_n\|_{\mathrm{L}^{\infty}(0,T;\V)}\|\v_n+\varepsilon\z_n\|_{\mathrm{L}^2(0,T;\D(\A))}\|\hat{ \z}_n\|^2_{\mathrm{L}^{\infty}(0,T;\V)} \\&+T\bigg[\|\v+\varepsilon\z\|^{2r-2}_{\mathrm{L}^{\infty}(0,T;\V)}+\|\v_n+\varepsilon\z_n\|^{2r-2}_{\mathrm{L}^{\infty}(0,T;\V)} +1\bigg]\|\hat{ \z}_n\|^2_{\mathrm{L}^{\infty}(0,T;\V)}+   \|\hat{ \f}_n\|^2_{\mathrm{L}^2(0,T; \H)}\\&+T^{\frac{1}{2}}\|\v+\varepsilon\z\|^2_{\mathrm{L}^{\infty}(0,T;\V)}\|\hat{ \z}_n\|_{\mathrm{L}^{\infty}(0,T;\V)}\|\hat{ \z}_n\|_{\mathrm{L}^2(0,T;\D(\A))}.
		\end{align*}
		Hence $\int_{0}^{T} \upbeta_n(s) \d s\to 0$ as $n\to \infty.$ Moreover, we have 
		\begin{align*}
		\int_{0}^{T}\upalpha_n(s)\d s&= T^{\frac{1}{2}}\|\v_n+\varepsilon\z_n\|_{\mathrm{L}^{\infty}(0,T;\V)}\|\v_n+\varepsilon\z_n\|_{\mathrm{L}^2(0,T;\D(\A))}+ T\|\v_n+\varepsilon\z_n\|^{2r-2}_{\mathrm{L}^{\infty}(0,T;\V)} \\&\quad+T\|\v+\varepsilon\z\|^{2r-2}_{\mathrm{L}^{\infty}(0,T;\V)}+T\|\v+\varepsilon\z\|^4_{\mathrm{L}^{\infty}(0,T;\V)}
		<\infty.
		\end{align*}
		Since, $\|\y_n(0)\|_{\V}=\|\u_0^n-\u_0\|_{\V}\to 0$ and $\int_{0}^{T} \upbeta_n(s) \d s\to 0$ as $n\to \infty$, and for all $n\in \mathbb{N}, \int_{0}^{T}\upalpha_n(s)\d s<\infty,$ then \eqref{Energy_esti_n_4} asserts that $\|\y_n(t)\|_{\V}\to 0$ as $n\to \infty$ uniformly in $t\in[0,T].$ Since $\v_n(\cdot)$ and $\v(\cdot)$ are continuous in $\V$, we also have $$\v_\varepsilon(\cdot, \z_n)\u_0^n \to \v_\varepsilon(\cdot, \z)\u_0\ \text{ in } \ \mathrm{C}([0,T];\V).$$ 
		From  \eqref{324}, we also infer that 
		\begin{align*}
		\mu \int_{0}^{T} \|\A\y_n(t)\|^2_{\H}&\leq \|\y_n(0)\|^2_{\V} + C\sup_{s\in[0,T]}\|\y_n(s)\|^2_{\V}\int_{0}^{T} \upalpha_n(s) \d s + C \int_{0}^{T} \upbeta_n(s)\ \d s\to 0,
		\end{align*}
		as $n\to \infty$ and therefore $$\v_\varepsilon(\cdot, \z_n)\u_0^n \to \v_\varepsilon(\cdot, \z)\u_0\  \text{ in } \ \mathrm{L}^2(0, T; \D(\A)),$$
		which completes the proof. 
	\end{proof}
	
	\begin{definition}
		We define a map $\varphi^{\alpha}_\varepsilon : \mathbb{R}^+ \times \Omega \times \V \to \V$ by
		\begin{align}
		(t, \omega, \u_0) \mapsto \v^{\alpha}_\varepsilon(t)  + \z_{\alpha}(\omega)(t) \in \V,
		\end{align}
		where $\v^{\alpha}_\varepsilon(t) = \v_\varepsilon(t, \z_{\alpha}(\omega)(t))(\u_0 - \z_{\alpha}(\omega)(0))$ is a solution to the problem \eqref{cscbf} with the initial condition $\u_0 - \z_{\alpha}(\omega)(0).$
	\end{definition}
	\begin{proposition}[Proposition 4.11, \cite{KM}]\label{alpha_ind}
		If $\alpha_1, \alpha_2 \geq 0$, then $\varphi^{\alpha_1}_\varepsilon = \varphi^{\alpha_2}_\varepsilon.$
	\end{proposition}
	Proposition \ref{alpha_ind} says that the map $\varphi^{\alpha}_{\varepsilon}$ does not depend on $\alpha$ and hence from now on, we will denote it  by $\varphi_{\varepsilon}$.
	\begin{theorem}
		$(\varphi, \theta)$ is an RDS.
	\end{theorem}
	\begin{proof}
		All the properties with the exception of the cocycle one of an RDS follow from Theorem \ref{RDS_Conti}. Hence we only need to show that for any $\u_0\in \V,$
		\begin{align}\label{Cocy..}
		\varphi(t+s, \omega)\u_0 = \varphi(t, \theta_s \omega)\varphi(s, \omega)\u_0, \ \ t, s \in \R^+. 
		\end{align}
		Remaining proof of the this theorem is similar to that of Theorem 6.15, \cite{BL} and hence we omit it here.
	\end{proof}
	Let us now define, for $\u_0 \in \V,\ \omega \in \Omega,$ and $t\geq s,$
	\begin{align}\label{combine_sol}
	\u(t, s;\omega, \u_0) := \varphi_\varepsilon(t-s; \theta_s \omega)\u_0 = \v\big(t, s; \omega, \u_0 - \z(s)\big) + \z(t),
	\end{align}
	then for each $s\in \mathbb{R}$ and each $\u_0 \in \V,$ the process $\u(t), \ t\geq s,$ is a solution to problem \eqref{S-CBF}.
	\section{Random attractors for 2D SCBF equations in $\H$}\label{sec4}\setcounter{equation}{0}
	The existence of random attractors in $\H$ for the 2D SCBF equations is established in this section. We consider the RDS $\varphi_{\varepsilon}$ over the metric DS $(\Omega, \hat{\mathscr{F}}, \hat{\mathbb{P}}, \hat{\theta})$.
	
	\begin{lemma}\label{Bddns4}
		For each $\omega\in \Omega,$ 
		\begin{align*}
		\limsup_{t\to - \infty} \|\z(\omega)(t)\|^2_{\H}\  e^{\mu \lambda_1 t +\frac{8}{ \mu} \int_{t}^{0}\|\z(\zeta)\|^{2}_{\V}\d\zeta} = 0.
		\end{align*}
	\end{lemma}
	\begin{proof}
		Let us fix $\omega\in \Omega$. By Corollary \ref{Bddns1}, we can find $t_0\leq 0$ such that for $t\leq t_0,$
		\begin{align}\label{Bddns3}
		\frac{8}{\mu} \int_{t}^{0} \|\z(s)\|^{2}_{\V} \d s < - \frac{\mu \lambda_1 t}{2}, \ \ \ t\leq t_0.
		\end{align}
		Since $\A$ is the generator of an analytic semigroup on $\V$, one can apply Proposition 2.11, \cite{BCLLLR}  find $\rho_1= \rho_1(\omega)\geq 0$ such that
		\begin{align}\label{rho}
		\frac{\|\z(t)\|_{\V}}{|t|} \leq \rho_1,  \  \text{ for } \  t\leq t_0.
		\end{align}
		Therefore, we have, for every $\omega\in \Omega,$
		\begin{align*}
		\limsup_{t\to - \infty} \|\z(\omega)(t)\|^2_{\H}\  e^{\mu \lambda_1 t +\frac{8}{\mu} \int_{t}^{0}\|\z(\zeta)\|^{2}_{\V}\d\zeta}\leq&\limsup_{t\to - \infty} \frac{1}{\lambda_1^2}\|\z(\omega)(t)\|^2_{\V}\  e^{\mu \lambda_1 t +\frac{8}{ \mu} \int_{t}^{0}\|\z(\zeta)\|^{2}_{\V}\d\zeta}\\\leq& \frac{\rho_1^2}{\lambda_1^2} \limsup_{t\to - \infty}  |t|^2 e^{\frac{\mu \lambda_1 t}{2} }
		=  0,
		\end{align*}
		which completes the proof.
	\end{proof}
	\begin{lemma}\label{Bddns5}
		For each $\omega\in \Omega,$
		\begin{align*}
		\int_{- \infty}^{0} \bigg\{ 1 + \|\z(t)\|^{2}_{\V} + \|\z(s)\|^{r+1}_{\V} + \|\z(t)\|^4_{\V}  \bigg\}e^{\mu \lambda_1 t +\frac{8}{\mu} \int_{t}^{0}\|\z(\zeta)\|^{4}_{\V}\d\zeta} \d t < \infty.
		\end{align*}
	\end{lemma}
	\begin{proof}Since for $t_0\leq 0$,
		\begin{align*}
		\int_{t_0}^{0} \bigg\{ 1 + \|\z(t)\|^{2}_{\V} + \|\z(s)\|^{r+1}_{\V} + \|\z(t)\|^4_{\V}  \bigg\}e^{\mu \lambda_1 t +\frac{8}{\mu} \int_{t}^{0}\|\z(\zeta)\|^{4}_{\V}\d\zeta} \d t < \infty,
		\end{align*}
		therefore, it is sufficient to prove that integral 
		\begin{align*}
		\int_{- \infty}^{t_0} \bigg\{ 1 + \|\z(t)\|^{2}_{\V} + \|\z(s)\|^{r+1}_{\V} + \|\z(t)\|^4_{\V}  \bigg\}e^{\mu \lambda_1 t +\frac{8}{\mu} \int_{t}^{0}\|\z(\zeta)\|^{4}_{\V}\d\zeta} \d t < \infty.
		\end{align*}
		Because of \eqref{Bddns3}, we obtain 
		\begin{align*}
		\int_{- \infty}^{t_0} e^{\mu \lambda_1 t +\frac{8}{\mu} \int_{t}^{0}\|\z(\zeta)\|^{2}_{\V}\d\zeta} \d t\leq \int_{- \infty}^{t_0} e^{\frac{\mu \lambda_1 t}{2}} \d t< \infty.
		\end{align*}
		Using \eqref{Bddns3} and \eqref{rho}, we deduce that 
		\begin{align*} 
		&	\int_{- \infty}^{t_0} \bigg\{ \|\z(t)\|^{2}_{\V} + \|\z(s)\|^{r+1}_{\V} + \|\z(t)\|^4_{\V}  \bigg\}e^{\mu \lambda_1 t +\frac{8}{\mu} \int_{t}^{0}\|\z(\zeta)\|^{4}_{\V}\d\zeta} \d t \\&\leq \int_{- \infty}^{t_0} \big\{ \rho^2_1|t|^2+\rho^{r+1}_1|t|^{r+1}+ \rho^4_1|t|^4  \big\}e^{\frac{\mu \lambda_1 t}{2} } \d t< \infty,
		\end{align*}
		which completes the proof.
	\end{proof}
	\begin{definition}\label{RA2}
		A function $\kappa: \Omega\to (0, \infty)$ belongs to class $\mathfrak{K}$ if and only if 
		\begin{align}
		\limsup_{t\to \infty} [\kappa(\theta_{-t}\omega)]^2 e^{-\mu \lambda_1 t +\frac{8}{\mu} \int_{-t}^{0}\|\z(\omega)(s)\|^{2}_{\V}\d s} = 0,
		\end{align}
		where $\lambda_1$ is the first eigenvalue of the Stokes operator $\A.$

		We denote by $\mathfrak{DK},$ the class of all closed and bounded random sets $\D$ on $\H$ such that the radius function $\Omega\ni \omega \mapsto \kappa(\D(\omega)):= \sup\{\|x\|_{\H}:x\in \D(\omega)\}$ belongs to the class $\mathfrak{K}.$
	\end{definition}
	By Corollary \ref{Bddns1}, we infer that the constant functions belong to $\mathfrak{K}$. 
	The class $\mathfrak{K}$ is closed with respect to sum, multiplication by a constant and if $\kappa \in \mathfrak{K}, 0\leq \bar{\kappa} \leq \kappa,$ then $\bar{\kappa}\in \mathfrak{K}.$
	\begin{proposition}\label{radius}
		Define  the functions $\kappa_{i}:\Omega\to (0, \infty), \ i= 1, 2, 3, 4, 5, 6,$ by the following formulae, for $\omega\in\Omega,$
		\begin{align*}
		[\kappa_1(\omega)]^2 &:= \|\z(\omega)(0)\|_{\H},\\
		[\kappa_2(\omega)]^2 &:= \sup_{s\leq 0} \|\z(\omega)(s)\|^2_{\H}\  e^{\mu \lambda_1 s +\frac{8}{ \mu} \int_{s}^{0}\|\z(\omega)(\zeta)\|^{2}_{\V}\d\zeta}, \\
		[\kappa_3(\omega)]^2 &:= \int_{- \infty}^{0} \|\z(\omega)(t)\|^{2}_{\V}\ e^{\mu \lambda_1 t +\frac{8}{ \mu} \int_{t}^{0}\|\z(\omega)(\zeta)\|^{2}_{\V}\d\zeta} \d t, \\
		[\kappa_4(\omega)]^2 &:= \int_{- \infty}^{0} \|\z(\omega)(t)\|^{r+1}_{\V}\ e^{\mu \lambda_1 t +\frac{8}{ \mu} \int_{t}^{0}\|\z(\omega)(\zeta)\|^{2}_{\V}\d\zeta} \d t,\\
		[\kappa_5(\omega)]^2 &:= \int_{- \infty}^{0} \|\z(\omega)(t)\|^4_{\V}\ e^{\mu \lambda_1 t +\frac{8}{ \mu} \int_{t}^{0}\|\z(\omega)(\zeta)\|^{2}_{\V}\d\zeta} \d t,\\
		[\kappa_6(\omega)]^2 &:= \int_{- \infty}^{0} e^{\mu \lambda_1 t +\frac{8}{ \mu} \int_{t}^{0}\|\z(\omega)(\zeta)\|^{2}_{\V}\d\zeta} \d t.
		\end{align*}
		Then all these functions belongs to the class $\mathfrak{K}.$
	\end{proposition}
	\begin{proof}Recall by Remark \ref{stationary} that $\z(\theta_{-t}\omega)(s) = \z(\omega)(s-t)$. Thus, we find  
		\begin{align*}
		\limsup_{t\to  \infty}[\kappa_1(\theta_{-t}\omega)]^2 e^{-\mu \lambda_1 t +\frac{8}{ \mu} \int_{-t}^{0}\|\z(\omega)(s)\|^{2}_{\V}\d s} =& \limsup_{t\to \infty}\|\z(\theta_{-t}\omega)(0)\|^2_{\H} e^{-\mu \lambda_1 t +\frac{8}{\mu} \int_{-t}^{0}\|\z(\omega)(s)\|^{2}_{\V}\d s}\\
		=&\limsup_{t\to \infty}\|\z(\omega)(-t)\|^2_{\H} e^{-\mu \lambda_1 t +\frac{8}{\mu} \int_{-t}^{0}\|\z(\omega)(s)\|^{2}_{\V}\d s}.
		\end{align*}
		Using Lemma \ref{Bddns4}, we have, $\kappa_1 \in \mathfrak{K}.$
		It can be easily seen that 
		\begin{align*}
		[\kappa_2(\theta_{-t}\omega)]^2 &=  \sup_{s\leq 0}  \|\z(\theta_{-t}\omega)(s)\|^2_{\H}\  e^{\mu \lambda_1 s +\frac{8}{\mu} \int_{s}^{0}\|\z(\theta_{-t}\omega)(\zeta)\|^{2}_{\V}\d\zeta}\\
		&=  \sup_{s\leq 0}  \|\z(\omega)(s-t)\|^2_{\H}\  e^{\mu \lambda_1 s +\frac{8}{ \mu} \int_{s}^{0}\|\z(\omega)(\zeta -t)\|^{2}_{\V}\d\zeta}\\
		&=  \sup_{s\leq 0}  \|\z(\omega)(s-t)\|^2_{\H}\  e^{\mu \lambda_1 (s-t) +\frac{8}{ \mu} \int_{s-t}^{-t}\|\z(\omega)(\zeta)\|^{2}_{\V}\d\zeta}\ e^{\mu \lambda_1 t}\\
		&=  \sup_{\sigma\leq -t}  \|\z(\omega)(\sigma)\|^2_{\H}\  e^{\mu \lambda_1 \sigma +\frac{8}{ \mu} \int_{\sigma}^{-t}\|\z(\omega)(\zeta)\|^{2}_{\V}\d\zeta}\ e^{\mu \lambda_1 t},
		\end{align*}and 
		\begin{align*}
		\limsup_{t\to \infty} [\kappa_2(\theta_{-t}\omega)]^2 e^{-\mu \lambda_1 t +\frac{8}{ \mu} \int_{-t}^{0}\|\z(\omega)(s)\|^{2}_{\V}\d s} &=\limsup_{t\to \infty} \sup_{\sigma\leq -t}  \|\z(\omega)(\sigma)\|^2_{\H}\  e^{\mu \lambda_1 \sigma +\frac{8}{ \mu} \int_{\sigma}^{0}\|\z(\omega)(\zeta)\|^{2}_{\V}\d\zeta}\\
		&=\limsup_{\sigma\to -\infty} \|\z(\omega)(\sigma)\|^2_{\H}\  e^{\mu \lambda_1 \sigma +\frac{8}{\mu} \int_{\sigma}^{0}\|\z(\omega)(\zeta)\|^{4}_{\widetilde{\L}^{4}}\d\zeta}\\
		& =0,
		\end{align*}
		using Lemma \ref{Bddns4}, which implies $\kappa_2\in \mathfrak{K}.$
		From the previous part of the proof, we infer that 
		\begin{align*}
		&\Big\{[\kappa_3(\theta_{-t}\omega)]^2+ [\kappa_4(\theta_{-t}\omega)]^2+ [\kappa_5(\theta_{-t}\omega)]^2 +[\kappa_6(\theta_{-t}\omega)]^2\Big\} e^{-\mu \lambda_1 t +\frac{8}{ \mu} \int_{-t}^{0}\|\z(\omega)(s)\|^{2}_{\V}\d s}\\
		&\quad=\int_{- \infty}^{-t} \bigg\{  \|\z(\omega)(t)\|^{2}_{\V} + \|\z(\omega)(s)\|^{r+1}_{\V} + \|\z(\omega)(t)\|^4_{\V} + 1 \bigg\}e^{\mu \lambda_1 \sigma +\frac{8}{ \mu} \int_{\sigma}^{0}\|\z(\omega)(\zeta)\|^{2}_{\V}\d\zeta} \d\sigma.
		\end{align*}
		Since by Lemma \ref{Bddns5}, we have 
		\begin{align*}
		\int_{- \infty}^{0} \bigg\{ \|\z(t)\|^{2}_{\V} + \|\z(s)\|^{r+1}_{\V} + \|\z(t)\|^4_{\V} +1  \bigg\}e^{\mu \lambda_1 t +\frac{8}{ \mu} \int_{t}^{0}\|\z(\zeta)\|^{2}_{\V}\d\zeta} \d t < \infty.
		\end{align*}
		By the Lebesgue monotone theorem, we conclude that as $t\to \infty$
		\begin{align*}
		\int_{- \infty}^{-t} \bigg\{  \|\z(t)\|^{2}_{\V} + \|\z(s)\|^{r+1}_{\V} + \|\z(t)\|^4_{\V} +1 \bigg\}e^{\mu \lambda_1 \sigma +\frac{8}{ \mu} \int_{\sigma}^{0}\|\z(\omega)(\zeta)\|^{2}_{\V}\d\zeta} \d\sigma \to 0.
		\end{align*}
		This implies that $\kappa_3, \kappa_4, \kappa_5, \kappa_6\in \mathfrak{K}$, which completes the proof.
	\end{proof}
	
	\begin{theorem}\label{H_ab}
		Assume that $\f\in\H$ and Assumption \ref{assump} holds. Then there exists a family $\hat{\mathrm{B}}_0=\{\mathrm{{ B}}_0(\omega):\omega\in \Omega\}$ of $\mathfrak{DK}$-random absorbing sets in $\H$ corresponding to the RDS $\varphi_\varepsilon.$
	\end{theorem}
	\begin{proof}
		Let ${\mathrm{D}}$ be a random set from the class $\mathfrak{DK}$. Let $\kappa_{\mathrm{D}}(\omega)$ be the radius of ${\mathrm{D}}(\omega)$, that is, $\kappa_{\mathrm{D}}(\omega):= \sup\{\|x\|_{\H} : x \in {\mathrm{D}}(\omega)\},\ \omega\in \Omega.$
		
		Let $\omega\in \Omega$ be fixed. For given $s\leq 0$ and $\u_0\in \H$, let $\v_\varepsilon(\cdot)$ be the unique weak solution of \eqref{cscbf} on time interval $[s, \infty)$ with the initial condition $\v_\varepsilon(s)= \u_0-\varepsilon\z(s).$
		Multiplying the first equation of \eqref{cscbf} by $\v_\varepsilon(\cdot)$ and integrating the resulting equation over $\mathcal{O}$, we obtain	
		\begin{align}\label{H_ab1}
		\frac{1}{2}\frac{\d}{\d t}\|\v_\varepsilon(t)\|^2_{\H} = & -\mu\|\v_\varepsilon(t)\|^2_{\V} -b(\v_\varepsilon(t)+\varepsilon\z(t),\v_\varepsilon(t)+\varepsilon\z(t),\v_\varepsilon(t))\nonumber\\&-\beta(\mathcal{C}\big(\v_\varepsilon(t)+\varepsilon\z(t)\big),\v_\varepsilon(t)) +\varepsilon\alpha\big(\z(t), \v_\varepsilon(t)\big) +\big(\f,\v_\varepsilon(t)\big)\nonumber\\
		=&-\mu\|\v_\varepsilon(t)\|^2_{\V} -\varepsilon b(\v_\varepsilon(t),\z(t),\v_\varepsilon(t))-\varepsilon^2b(\z(t),\z(t),\v_\varepsilon(t))\nonumber\\&+\varepsilon\beta(\mathcal{C}\big(\v_\varepsilon(t)+\varepsilon\z(t)\big),\z(t))-\beta(\mathcal{C}\big(\v_\varepsilon(t)+\varepsilon\z(t)\big),\v_\varepsilon(t)+\varepsilon\z(t)) \nonumber\\&+\varepsilon\alpha\big(\z(t), \v_\varepsilon(t)\big) +\big(\f,\v_\varepsilon(t)\big) .
		\end{align}
		Using $0<\varepsilon\leq1$, H\"older's inequality, Young's inequality, Sobolev's embedding, \eqref{poin} and \eqref{lady}, we have
		\begin{align*}
		|\varepsilon b(\v_\varepsilon,\z, \v_\varepsilon)|&\leq \|\v_\varepsilon\|^2_{\widetilde{\L}^{4}}\|\z\|_{\V}\leq\sqrt{2}\|\v_\varepsilon\|_{\H}\|\v_\varepsilon\|_{\V}\|\z\|_{\V}\leq \frac{\mu}{8} \|\v_\varepsilon\|^2_{\V}+ \frac{4}{\mu}\|\z\|^2_{\V}\|\v_\varepsilon\|^2_{\H} ,\\
		|\varepsilon^2b(\z,\z, \v_\varepsilon)|&\leq \varepsilon^2\|\z\|^2_{\widetilde{\L}^{4}}\|\v_\varepsilon\|_{\V}\leq\varepsilon^2\sqrt{2}\|\z\|_{\H}\|\z\|_{\V}\|\v_\varepsilon\|_{\V}\leq \frac{\mu}{8} \|\v_\varepsilon\|^2_{\V}+ \frac{4\varepsilon^4}{\mu}\|\z\|^2_{\H}\|\z\|^2_{\V},\\
		\beta(\mathcal{C}(\v_\varepsilon+\varepsilon\z),\v_\varepsilon+\varepsilon\z) & =  \beta\|\v_\varepsilon+\varepsilon\z\|^{r+1}_{\widetilde{\L}^{r+1}},\\
		|\varepsilon\beta\big\langle\mathcal{C}(\v_\varepsilon+\varepsilon\z),\z\big\rangle|& \leq \varepsilon\beta \|\v_\varepsilon+\varepsilon\z\|^{r}_{\widetilde{\L}^{r+1}} \|\z\|_{\widetilde{\L}^{r+1}} \leq \frac{\beta}{2} \|\v_\varepsilon+\varepsilon\z\|^{r+1}_{\widetilde{\L}^{r+1}} + \frac{\varepsilon^{r+1}\beta(2r)^r}{(r+1)^{r+1}}\|\z\|^{r+1}_{\widetilde{\L}^{r+1}}\\
		& \leq \frac{\beta}{2} \|\v_\varepsilon+\varepsilon\z\|^{r+1}_{\widetilde{\L}^{r+1}} + \varepsilon^{r+1}C\|\z\|^{r+1}_{\V},\\
		|\varepsilon\alpha\big( \z, \v_\varepsilon\big) |  &\leq \varepsilon\alpha\|\z\|_{\H} \|\v_\varepsilon\|_{\H}\leq \frac{\varepsilon\alpha}{\lambda_1^2}\|\z\|_{\V} \|\v_\varepsilon\|_{\V}\leq \frac{\mu}{8} \|\v_\varepsilon\|^2_{\V} + \frac{4\varepsilon^2\alpha^2}{\mu\lambda_1^4} \|\z\|^2_{\V},\\
		|\big(\f, \v_\varepsilon\big) | &\leq  \|\f\|_{\H} \|\v_\varepsilon\|_{\H}\leq \frac{1}{\lambda_1}\|\f\|_{\H} \|\v_\varepsilon\|_{\V}\leq  \frac{\mu}{8} \|\v_\varepsilon\|^2_{\V} + \frac{4}{\mu\lambda_1^2} \|\f\|^2_{\H}.
		\end{align*}
		Thus from \eqref{H_ab1}, we deduce that
		\begin{align*}
		&\frac{1}{2}\frac{\d}{\d t} \|\v_\varepsilon(t)\|^2_{\H}  + \frac{\mu}{2} \|\v_\varepsilon(t)\|^2_{\V}  +\frac{\beta}{2}\|\v_\varepsilon(t)+\varepsilon\z(t)\|^{r+1}_{\widetilde{\L}^{r+1}}\nonumber\\&\leq  \frac{4 }{\mu} \|\v_\varepsilon(t)\|^2_{\H}\ \|\z(t)\|^{2}_{\V} +\frac{4\varepsilon^4}{\mu\lambda_1^2}\|\z(t)\|^4_{\V} + C\varepsilon^{r+1}\|\z(t)\|^{r+1}_{\V}+  \frac{4\alpha^2\varepsilon^2}{\mu\lambda_1^4} \|\z(t)\|^2_{\V} + \frac{4}{\mu\lambda_1^2} \|\f\|^2_{\H},
		\end{align*}
		and 
		\begin{align}\label{H_ab2}
		&\frac{\d}{\d t} \|\v_\varepsilon(t)\|^2_{\H}  + \mu\lambda_1 \|\v_\varepsilon(t)\|^2_{\H}  \nonumber\\&\leq  \frac{8}{\mu} \|\v_\varepsilon(t)\|^2_{\H}\ \|\z(t)\|^{2}_{\V} +\frac{8\varepsilon^4}{\mu\lambda_1^2}\|\z(t)\|^4_{\V} + \varepsilon^{r+1}C\|\z(t)\|^{r+1}_{\V}+  \frac{8\alpha^2\varepsilon^2}{\mu\lambda_1^4} \|\z(t)\|^2_{\V} + \frac{8}{\mu\lambda_1^2} \|\f\|^2_{\H},
		\end{align}
	for a.e. $t\in[0,T]$.	We infer from  the classical Gronwall inequality that 
		\begin{align}\label{H_ab3}
		\|\v_\varepsilon(0)\|^2_{\H} &\leq \|\v_\varepsilon(s)\|^2_{\H} e^{\mu\lambda_1 s + \frac{8}{\mu}\int_{s}^{0}\|\z(\zeta)\|^2_{\V}\d\zeta} +\int_{s}^{0}\biggl\{\frac{8\varepsilon^4}{\mu\lambda_1^2}\|\z(t)\|^4_{\V}  + C\varepsilon^{r+1}\|\z(t)\|^{r+1}_{\V}\nonumber\\& \quad+  \frac{8\alpha^2\varepsilon^2}{\mu\lambda_1^4} \|\z(t)\|^2_{\V}  + \frac{8\varepsilon^2}{\mu\lambda_1^2} \|\f\|^2_{\H}\biggr\}e^{\mu\lambda_1t  + \frac{8}{\mu}\int_{t}^{0}\|\z(\zeta)\|^2_{\V}\d\zeta}\d t\nonumber\\
		&\leq2 \|\u_0\|^2_{\H} e^{\mu\lambda_1 s + \frac{8}{\mu}\int_{s}^{0}\|\z(\zeta)\|^2_{\V}\d\zeta} +2\varepsilon^2\|\z(s)\|^2_{\H} e^{\mu\lambda_1 s + \frac{8}{\mu}\int_{s}^{0}\|\z(\zeta)\|^2_{\V}\d\zeta} +\int_{s}^{0}\bigg\{\frac{8\varepsilon^4}{\mu\lambda_1^2}\|\z(t)\|^4_{\V} \nonumber \\& \quad + C\varepsilon^{r+1}\|\z(t)\|^{r+1}_{\V}+  \frac{8\alpha^2\varepsilon^2}{\mu\lambda_1^4} \|\z(t)\|^2_{\V}  + \frac{8}{\mu\lambda_1^2} \|\f\|^2_{\H}\biggr\}e^{\mu\lambda_1t  + \frac{8}{\mu}\int_{t}^{0}\|\z(\zeta)\|^2_{\V}\d\zeta}\d t.
		\end{align}
		Let us set for $\omega\in \Omega,$
		\begin{align}
		[\kappa_{11}(\omega)]^2  &=2 +  2\varepsilon^2\sup_{s\leq 0}\bigg\{ \|\z(s)\|^2_{\H}\  e^{\mu \lambda_1 s +\frac{8}{ \mu} \int_{s}^{0}\|\z(\zeta)\|^{2}_{\V}\d\zeta}\bigg\} + \int_{- \infty}^{0} \bigg\{\frac{8\varepsilon^4}{\mu\lambda_1^2}\|\z(t)\|^4_{\V}  \nonumber\\& \quad + C\varepsilon^{r+1}\|\z(t)\|^{r+1}_{\V}+  \frac{8\alpha^2\varepsilon^2}{\mu\lambda_1^4} \|\z(t)\|^2_{\V}  + \frac{8}{\mu\lambda_1^2} \|\f\|^2_{\H}\bigg\}e^{\mu \lambda_1 t +\frac{8}{ \mu} \int_{t}^{0}\|\z(\zeta)\|^{2}_{\V}\d\zeta} \d t,\\	
		\kappa_{12}(\omega)&= \varepsilon \|\z(\omega)(0)\|_{\H}.
		\end{align}
		By Lemma \ref{Bddns5} and Proposition \ref{radius}, we infer that both $\kappa_{11}$ and $\kappa_{12}$ belong to the class $\mathfrak{K}$ and also that $\kappa_{13}:=\kappa_{11}+\kappa_{12}$ belongs to the class $\mathfrak{K}$ as well. Therefore the random set $\text{B}_0$ defined by $$\text{B}_0(\omega) := \{\u_\varepsilon\in\H: \|\u_\varepsilon\|_{\H}\leq \kappa_{13}(\omega)\},$$ belongs to the family $\mathfrak{DK}.$
		
		We will show now that $\text{B}_0$ absorbs ${\mathrm{D}}$. Let $\omega\in\Omega$ be fixed. Since $\kappa_{\mathrm{D}}(\omega)\in \mathfrak{K}$, there exists $t_{\mathrm{D}}(\omega)\geq 0$ such that 
		\begin{align*}
		[\kappa_{\mathrm{D}}(\theta_{-t}\omega)]^2 e^{-\mu \lambda_1 t +\frac{8}{ \mu} \int_{-t}^{0}\|\z(\omega)(s)\|^{2}_{\V}\d s} \leq 1 \  \text{ for }\  t\geq t_{\mathrm{D}}(\omega). 
		\end{align*}
		Thus, if $\u_0\in {\mathrm{D}}(\theta_{-t}\omega)$ and $s\leq- t_{\mathrm{D}}(\omega),$ then for any $\varepsilon\in(0,1]$ by \eqref{H_ab3}, we get
		\begin{align}\label{ab_H}
		\|\v_\varepsilon(0,s;\omega, \u_0-\z(s))\|_{\H}\leq \kappa_{11}(\omega).	
		\end{align}
		Moreover, we have 
		\begin{align}\label{H_ab5}
		\|\u_\varepsilon(0,s;\omega, \u_0)\|_{\H} \leq \|\v_\varepsilon(0,s;\omega, \u_0-\varepsilon\z(s))\|_{\H} + \varepsilon\|\z(\omega)(0)\|_{\H}\leq \kappa_{13}(\omega).
		\end{align}
		This implies that $\u_\varepsilon(0,s;\omega; \u_0) \in \text{B}_0(\omega)$, for all $s\leq -t_{\mathrm{D}}(\omega).$ This proves that $\text{B}_0$ absorbs ${\mathrm{D}}$.
		
		Furthermore, integrating \eqref{H_ab2} over $(-1,0)$, we find for any $\omega\in \Omega$, there exists a $\kappa_{14}(\omega)\geq 0$ such that
		\begin{align}\label{H_ab4}
		\int_{-1}^{0} \bigg[\|\v_\varepsilon(t)\|^2_{\V} + \|\v_\varepsilon(t)+ \varepsilon\z(t)\|^{r+1}_{\widetilde{\L}^{r+1}} \bigg]\d t \leq \kappa_{14}(\omega),
		\end{align}
		for all $s\leq -t_{\mathrm{D}}(\omega).$
	\end{proof}
	\begin{theorem}\label{V_ab}
		Assume that $\f\in\H$ and Assumption \ref{assump} holds. Then there exists a family $\hat{\mathrm{B}}_1=\{\mathrm{{ B}}_1(\omega):\omega\in \Omega\}$ of $\mathfrak{DK}$-random absorbing sets in $\V$ corresponding to the RDS $\varphi_\varepsilon.$
	\end{theorem}
	\begin{proof}
		Let $\omega\in \Omega$ be fixed. For given $s\leq 0$ and $\u_0\in \H$, let $\v_\varepsilon(\cdot)$ be the unique solution of \eqref{cscbf} on time interval $[s, \infty)$ with the initial condition $\v_\varepsilon(s)= \u_0-\varepsilon\z(s).$
		Multiplying first equation of \eqref{cscbf} by $\A\v_\varepsilon(\cdot)$ and then integrating the resulting equation over $\mathcal{O}$, we obtain
		\begin{align}\label{V_ab1}
		\frac{1}{2}\frac{\d}{\d t}\|\v_\varepsilon(t)\|^2_{\V} +\mu\|\A\v_\varepsilon(t)\|^2_{\H}= &  -b(\v_\varepsilon(t)+\varepsilon\z(t),\v_\varepsilon(t)+\varepsilon\z(t),\A\v_\varepsilon(t))\nonumber\\&-\beta(\mathcal{C}\big(\v_\varepsilon(t)+\varepsilon\z(t)\big),\A\v_\varepsilon(t))  +\varepsilon\alpha\big(\z(t), \A\v_\varepsilon(t)\big) \nonumber\\&+\big(\f,\A\v_\varepsilon(t)\big).
		\end{align}
		Using $0<\varepsilon\leq1$, \eqref{b1}, H\"older's and Young's inequalities, we estimate
		\begin{align}
		\big|b(\v_\varepsilon+\varepsilon\z,\v_\varepsilon+\varepsilon\z,\A\v_\varepsilon)\big|&\leq C \|\v_\varepsilon+\varepsilon\z\|^{\frac{1}{2}}_{\H}\|\v_\varepsilon+\varepsilon\z\|_{\V} \|\A\v_\varepsilon+\varepsilon\A\z\|^{\frac{1}{2}}_{\H}\|\A\v_\varepsilon\|_{\H} \nonumber\\
		&\leq C \|\v_\varepsilon+\varepsilon\z\|^{\frac{1}{2}}_{\H}\|\v_\varepsilon+\varepsilon\z\|_{\V} \big(\|\A\v_\varepsilon\|^{\frac{1}{2}}_{\H}+\varepsilon^{\frac{1}{2}}\|\A\z\|^{\frac{1}{2}}_{\H}\big)\|\A\v_\varepsilon\|_{\H} \nonumber\\
		&\leq\frac{\mu}{8}\|\A\v_\varepsilon\|^2_{\H} + C\|\v_\varepsilon+\varepsilon\z\|_{\H}^2\|\v_\varepsilon+\varepsilon\z\|^4_{\V}\nonumber\\&\quad+\varepsilon C\|\v_\varepsilon+\varepsilon\z\|_{\H}\|\v_\varepsilon+\varepsilon\z\|^2_{\V} \|\A\z\|_{\H}\nonumber\\
		&\leq\frac{\mu}{8}\|\A\v_\varepsilon\|^2_{\H} +C\|\v_\varepsilon+\varepsilon\z\|_{\H}^2\|\v_\varepsilon\|^4_{\V}+\varepsilon C\|\v_\varepsilon+\varepsilon\z\|_{\H} \|\A\z\|_{\H}\|\v_\varepsilon\|^2_{\V}\nonumber\\
		&\quad +\varepsilon^4C\|\v_\varepsilon+\varepsilon\z\|_{\H}^2\|\z\|^4_{\V}+\varepsilon^3C\|\v_\varepsilon+\varepsilon\z\|_{\H} \|\A\z\|_{\H}\|\z\|^2_{\V},\label{V_ab2}\\
		\big|(\mathcal{C}\big(\v_\varepsilon+\varepsilon\z\big),\A\v_\varepsilon)\big|&\leq \|\v_\varepsilon+\varepsilon\z\|^{r}_{\wi \L^{2r}}\|\A\v_\varepsilon\|_{\H}\leq \frac{\mu}{8}\|\A\v_\varepsilon\|^2_{\H} + C\|\v_\varepsilon+\varepsilon\z\|^{2r}_{\wi \L^{2r}}\nonumber\\&\leq \frac{\mu}{8}\|\A\v_\varepsilon\|^2_{\H} + C\|\v_\varepsilon+\varepsilon\z\|^{2r-2}_{\H}\|\v_\varepsilon+\varepsilon\z\|^2_{\V}\nonumber\\&\leq \frac{\mu}{8}\|\A\v_\varepsilon\|^2_{\H} + C\|\v_\varepsilon+\varepsilon\z\|^{2r-2}_{\H}\|\v_\varepsilon\|^2_{\V}+\varepsilon^2C\|\v_\varepsilon+\varepsilon\z\|^{2r-2}_{\H}\|\z\|^2_{\V},\label{V_ab3}\\
		\varepsilon\alpha|\big(\z,\A\v_\varepsilon\big)|&\leq\varepsilon \alpha \|\z\|_{\H}\|\A\v_\varepsilon\|_{\H}	\leq\frac{\mu}{8}\|\A\v_\varepsilon\|^2_{\H} + \varepsilon^2C\|\z\|^2_{\H}\leq\frac{\mu}{8}\|\A\v_\varepsilon\|^2_{\H} + \varepsilon^2C\|\z\|^2_{\V},\label{V_ab4}\\
		|	\big(\f, \A\v_\varepsilon\big)|&\leq \|\f\|_{\H}\|\A\v_\varepsilon\|_{\H}\leq\frac{\mu}{8}\|\A\v_\varepsilon\|^2_{\H} + C\|\f\|^2_{\H}.\label{V_ab5}
		\end{align}
		We infer from the inequalities \eqref{V_ab2}-\eqref{V_ab5} that for any $\omega\in\Omega$,
		\begin{align}\label{V_ab7}
		\frac{\d}{\d t}\|\v_\varepsilon(t)\|^2_{\V} +\mu\|\A\v_\varepsilon(t)\|^2_{\H} 
		&\leq\bigg[C\|\v_\varepsilon(t)+\varepsilon\z(t)\|_{\H}^2\|\v_\varepsilon(t)\|^2_{\V} +\varepsilon C\|\v_\varepsilon(t)+\varepsilon\z(t)\|_{\H}\|\A\z(t)\|_{\H}\nonumber\\&\quad+C\|\v_\varepsilon(t)+\varepsilon\z(t)\|^{2r-2}_{\H} \bigg]\|\v_\varepsilon(t)\|^2_{\V} +\varepsilon^4 C\|\v_\varepsilon(t)+\varepsilon\z(t)\|_{\H}^2\|\z(t)\|^4_{\V}\nonumber\\&\quad+\varepsilon^3C\|\v_\varepsilon(t)+\varepsilon\z(t)\|_{\H} \|\A\z(t)\|_{\H}\|\z(t)\|^2_{\V} \nonumber\\&\quad+\varepsilon^2 C\|\v_\varepsilon(t)+\varepsilon\z(t)\|^{2r-2}_{\H}\|\z(t)\|^2_{\V} +\varepsilon^2C\|\z(t)\|^2_{\V}+C\|\f\|^2_{\H}.
		\end{align}
		Thus, it is immediate that 
		\begin{align}
		\frac{\d}{\d t}\|\v_\varepsilon(t)\|^2_{\V} 
		&\leq S_1(t)\|\v_\varepsilon(t)\|^2_{\V} +S_2(t),\label{V_ab8}
		\end{align}
		where
		\begin{align}
		S_1&=	C\|\v_\varepsilon+\varepsilon\z\|_{\H}^2\|\v_\varepsilon\|^2_{\V} +\varepsilon C\|\v_\varepsilon+\varepsilon\z\|_{\H}\|\A\z\|_{\H}+C\|\v_\varepsilon+\varepsilon\z\|^{2r-2}_{\H}, \label{V_ab9}\\
		S_2&= \varepsilon^4 C\|\v_\varepsilon+\varepsilon\z\|_{\H}^2\|\z\|^4_{\V}+\varepsilon^3C\|\v_\varepsilon+\varepsilon\z\|_{\H} \|\A\z\|_{\H}\|\z\|^2_{\V}+\varepsilon^2 C\|\v_\varepsilon+\varepsilon\z\|^{2r-2}_{\H}\|\z\|^2_{\V}  \nonumber\\&\quad+\varepsilon^2C\|\z\|^2_{\V}+C\|\f\|^2_{\H}.\label{V_ab10}
		\end{align}
		From Theorem \ref{H_ab}, we get for any $\varepsilon\in(0,1]$ and for all $t\in[-1,0]$,
		\begin{align}\label{H_ab6}
		\|\v_\varepsilon(t,s;\omega,\u_0-\z(s))\|_{\H}\leq\kappa_{11}(\omega),
		\end{align}
		for any $s\leq-(t_{\mathrm{D}}(\omega)+1).$ Therefore, for $s\leq-(t_{\D}(\omega)+1)$,  using \eqref{O-U_conti},  \eqref{H_ab4} and \eqref{H_ab6}, we obtain  
		$$\int_{-1}^{0}\|\v_\varepsilon(t)\|^2_{\V}\d t<\infty,\quad\int_{-1}^{0}S_1(t)\d t<\infty, \quad\int_{-1}^{0}S_2(t)\d t<\infty.$$
		Hence, by the uniform Gronwall lemma (Lemma 1.1 \cite{R.Temam}), we infer that for any $\varepsilon\in(0,1]$ and for any $\omega\in\Omega,$ there exists $\kappa_{15}(\omega)\geq0$ such that 
		\begin{align}\label{ab_V}
		\|\v_\varepsilon(0,\omega; s, \u_0-\z(s))\|_{\V}\leq \kappa_{15}(\omega),
		\end{align}
		for any $s\leq-(t_{\mathrm{D}}(\omega)+1)$.  Moreover, integrating \eqref{V_ab7} over $(-1,0)$, we find for any $\varepsilon\in(0,1]$ and for any $\omega\in \Omega$, there exists $\kappa_{16}(\omega)\geq 0$ such that
		\begin{align}\label{V_ab_12}
		\mu\int_{-1}^{0}&\|\A\v_\varepsilon(t)\|^2_{\H} \d t \leq \kappa_{16}(\omega),
		\end{align}
		for any $s\leq-(t_{\mathrm{D}}(\omega)+1).$
	\end{proof}
	Thanks to the compactness of $\V$ in $\H$, from Theorems \ref{H_ab} and  \ref{V_ab} and the abstract theory of random attractors (Theorem 3.11, \cite{CF}), we immediately conclude the following result.
	
	\begin{theorem}\label{Main_Theoem_H}
		Suppose that $\f\in\H$ and Assumption \ref{assump} holds. Then the cocycle $\varphi_\varepsilon$ corresponding to the 2D SCBF equations with small additive noise \eqref{S-CBF} has a random attractor $\hat{\mathcal{A}}_{\varepsilon} = \{\mathbf{A}_{\varepsilon}(\omega): \omega\in \Omega\}$ in $\H$. 
	\end{theorem}

	\section{Upper semi-continuity of $\mathfrak{DK}$-random attractors in $\H$}\label{sec5}\setcounter{equation}{0}
	In this section, we prove the upper semicontinuity of random attractors in $\H$. The existence of random attractors for the stochastic system \eqref{S-CBF} in $\H$ is proved in Theorem \ref{Main_Theoem_H} and the existence of global attractors for the deterministic system \eqref{D-CBF}  in $\H$ is established in Theorem 3.7, \cite{MTM1}. Upper semicontinuity results for the stochastic Navier-Stokes equations is obtained in \cite{CLR} and stochastic Cahn-Hilliard-Navier-Stokes system is established in \cite{FLYB}. Now, using the similar techniques in the work \cite{CLR}, we state and prove the following theorem on the upper semicontinuity of the random attractors: 
	\begin{theorem}\label{USC}
		Suppose that $\f\in\H$ and Assumption \ref{assump} is satisfied. Also, assume that the deterministic system \eqref{D-CBF} has a global attractor $\hat{\mathcal{A}}$ and its small random perturbed dynamical system \eqref{S-CBF} possesses a $\mathfrak{DK}$-random attractor $\hat{\mathcal{A}}_\varepsilon= \{\mathbf{A}_\varepsilon(\omega):\omega\in \Omega\},$ for any $\varepsilon\in (0, 1].$ If the following conditions hold:
		\begin{itemize}
			\item [$(K_1)$] For each $t_0\geq0$ and for $\hat{\mathbb{P}}$-a.e. $\omega\in\Omega$ $$\lim_{\varepsilon\to0^+} d\big(\varphi_\varepsilon(t_0, \theta_{-t_0}\omega)\u_0,\S(t)\u_0\big)=0,$$
			uniformly on bounded sets of $\H$, where $\varphi_\varepsilon$ is a RDS and $\S(t)$ is a semigroup generated by \eqref{S-CBF} and \eqref{D-CBF}, respectively with same initial condition $\u_0$.  
			\item[$(K_2)$] There exists a compact set $K\subset \H$ such that 
			$$\lim_{\varepsilon\to 0^+} d(\mathbf{A}_{\varepsilon}(\omega), K)=0,$$for $\hat{\mathbb{P}}$-a.e. $\omega\in\Omega.$
		\end{itemize}
		Then $\hat{\mathcal{A}}_\varepsilon$ and $\hat{\mathcal{A}}$ have the property of upper semi-continuity, that is,\begin{align}\label{51}\lim_{\varepsilon\to 0^+} d(\hat{\mathcal{A}}_{\varepsilon}(\omega), \hat{\mathcal{A}})=0,\end{align}for $\hat{\mathbb{P}}$-a.e. $\omega\in\Omega.$

		Furthermore, if for $\varepsilon_0\in (0,1]$ we have that for $\hat{\mathbb{P}}$-a.e. $\omega\in\Omega$ and all $t_0>0$ \begin{align}\label{K_3}
		\varphi_{\varepsilon}(t_0, \theta_{-t_0}\omega)\u_0 \to \varphi_{\varepsilon_0}(t_0, \theta_{-t_0}\omega)\u_0\ \text{ as }\ \varepsilon\to\varepsilon_0,
		\end{align} uniformly on bounded sets of $\H$, then the convergence \eqref{K_3} is upper semicontinuous in $\varepsilon$, that is, \begin{align}\label{52}
		\lim_{\varepsilon\to\varepsilon_0} d(\hat{\mathcal{A}}_{\varepsilon}(\omega),\hat{\mathcal{A}}_{\varepsilon_0}(\omega)) = 0,
		\end{align}for $\hat{\mathbb{P}}$-a.e. $\omega\in\Omega.$
		
	\end{theorem}
	\begin{proof}
		To prove the property of upper semicontinuity for our system, we only need to verify the conditions $(K_1)$ and $(K_2)$.
		\vskip 0.2cm
		\noindent\textbf{Step I.} \emph{Verification of $(K_2)$:} Let us introduce $$\v_\varepsilon(t,\omega)=\u_\varepsilon(t, \omega)-\varepsilon\z(t, \omega),$$ where $\u_\varepsilon(t, \omega)$ and $\z(t,\omega)$ are the unique solutions of \eqref{S-CBF} and \eqref{OUPe}, respectively. Also from \eqref{O-U_conti}, we have $\z\in\mathrm{L}^{\infty}_{loc}([t_0,\infty);\V)\cap\mathrm{L}^2_{loc}([t_0,\infty);\D(\A))$. Clearly, $\v_\varepsilon$ satisfies
		\begin{equation}\label{cscbf_s}
		\begin{aligned}
		\frac{\d\v_\varepsilon}{\d t} &= -\mu \A\v_\varepsilon - \B(\v_\varepsilon + \varepsilon\z) - \beta \mathcal{C}(\v_\varepsilon + \varepsilon\z) + \varepsilon\alpha \z + \f. 
		\end{aligned}
		\end{equation}
		From Theorems \ref{H_ab} and \ref{V_ab}, we observe that there exists $\hat{\kappa}_\varepsilon(\omega)\in\mathfrak{K}$-class such that $$\|\v_\varepsilon(0)\|_{\V}\leq\hat{\kappa}_\varepsilon(\omega).$$ If we call $K_\varepsilon(\omega),$ the ball in $\V$ of radius $\hat{\kappa}_\varepsilon(\omega)+\varepsilon\|\z(0)\|_{\V}$, we have a compact (since $\V\hookrightarrow\H$ is compact) $\mathfrak{DK}$-absorbing set in $\H$ for $\varphi_\varepsilon$. Furthermore, there exists a $\hat{\kappa}_d$ independent of $\omega\in \Omega$ such that $$\lim_{\varepsilon\to0^+} \hat{\kappa}_\varepsilon(\omega)\leq\hat{\kappa}_d,$$which verifies  Lemma 1, \cite{CLR} and hence $(K_2)$ follows.
		\vskip 0.2cm
		\noindent
		\textbf{Step II.} \emph{Verification of $(K_1)$:} In order to verify the assertion $(K_1)$, it is enough to prove that the solution $\varphi_\varepsilon(t,\omega)\u_0$ of system \eqref{S-CBF} $\hat{\mathbb{P}}$-a.s. converges to the solution $\S(t)\u_0$ of the unperturbed system \eqref{D-CBF} in $\H$ as $\varepsilon\to 0^+$ uniformly on bounded sets of initial conditions, that is, for $\hat{\mathbb{P}}$-a.e $\omega\in\Omega$, any $t_0\geq0$ and any bounded subset $\G\subset\H$, we have 
		
		\begin{align}\label{usc}
		\lim_{\varepsilon\to 0^+}\sup_{\u_0\in \G} \|\varphi_\varepsilon(t_0,\theta_{-t_0}\omega)\u_0-\S(t_0)\u_0\|_{\H}=0.
		\end{align}
		For any $\u_0\in \G,$ let $\u_\varepsilon(t)=\varphi_{\varepsilon}(t+t_0,\theta_{-t_0})\u_0$ and $\u(t)=\S(t+t_0)\u_0$ respectively, be the unique weak solutions of the systems \eqref{S-CBF} and \eqref{D-CBF} with initial condition $\u_0$ at $t=-t_0$. Also, let for $T\geq 0$,$$\y_\varepsilon(t)= \u_\varepsilon(t)-\u(t), \ \ t\in[-t_0,T].$$ Clearly $\y_\varepsilon(\cdot)$ satisfies
		\begin{equation}\label{cSCBF_y}
		\left\{
		\begin{aligned}
		\d\y_\varepsilon+\{\mu \A\y_\varepsilon+\B(\y_\varepsilon+\u)-\B(\u)+\beta \mathcal{C}(\y_\varepsilon+\u)-\beta\mathcal{C}(\u)\}\d t&= \varepsilon\d\text{W}(t),  \\ 
		\y_\varepsilon(-t_0)&=0,
		\end{aligned}
		\right.
		\end{equation}
		in $\V'$  for all $ t\in[-t_0,T]$.	Let us introduce $\eta_\varepsilon(\cdot)=\y_\varepsilon(\cdot)-\varepsilon\z(\cdot)$, where $\z(\cdot)$ is the solution of \eqref{OUPe}, then $\eta_\varepsilon(\cdot)$ satisfies the following equation in $\V'$:
		\begin{equation}\label{cscbf_eta1}
		\left\{
		\begin{aligned}
		\frac{\d\eta_\varepsilon}{\d t} &= -\mu \A\eta_\varepsilon - \B(\eta_\varepsilon + \varepsilon\z+\u)+\B(\u) - \beta \mathcal{C}(\eta_\varepsilon + \varepsilon\z +\u)+\beta\mathcal{C}(\u) +\varepsilon \alpha \z, \\
		\eta_\varepsilon(-t_0)&= -\varepsilon \z(-t_0).
		\end{aligned}
		\right.
		\end{equation}
		or
		\begin{equation}\label{cscbf_eta2}
		\left\{
		\begin{aligned}
		\frac{\d\eta_\varepsilon}{\d t} &= -\mu \A\eta_\varepsilon - \B(\eta_\varepsilon,\eta_\varepsilon)-\varepsilon\B(\eta_\varepsilon,\z)-\B(\eta_\varepsilon,\u)-\varepsilon\B(\z,\eta_\varepsilon)-\varepsilon^2\B(\z,\z)-\varepsilon(\z,\u)\\&\quad-\B(\u,\eta_\varepsilon)-\varepsilon\B(\u,\z) - \beta \mathcal{C}(\eta_\varepsilon + \varepsilon\z +\u)+\beta\mathcal{C}(\u) +\varepsilon \alpha \z, \\
		\eta_\varepsilon(-t_0)&= -\varepsilon \z(-t_0).
		\end{aligned}
		\right.
		\end{equation}
		Taking the inner product of the first equation of \eqref{cscbf_eta2} with $\eta_\varepsilon(\cdot)$ in $\H$ and making use of \eqref{b0}, we get 
		\begin{align}\label{usc1}
		\frac{1}{2}\frac{\d}{\d t}\|\eta_\varepsilon(t)\|^2_{\H} =& - \mu\|\eta_\varepsilon(t)\|^2_{\V} -\varepsilon b(\eta_\varepsilon(t),\z(t),\eta_\varepsilon(t))-b(\eta_\varepsilon(t),\u(t),\eta_\varepsilon(t))\nonumber\\&+\varepsilon^2b(\z(t),\eta_\varepsilon(t),\z(t))+\varepsilon b(\z(t),\eta_\varepsilon(t),\u(t))+\varepsilon b(\u(t),\eta_\varepsilon(t),\z(t)) \nonumber\\& -\beta\big\langle\mathcal{C}(\eta_\varepsilon(t)+\varepsilon\z(t)+\u(t))-\mathcal{C}(\u(t)),\eta_\varepsilon(t)+\varepsilon\z(t)+\u(t)-\u(t)\big\rangle \nonumber\\&+ \varepsilon\beta\big\langle\mathcal{C}(\eta_\varepsilon(t)+\varepsilon\z(t)+\u(t))-\mathcal{C}(\u(t)),\z(t)\big\rangle+\varepsilon\alpha\langle\z(t),\eta_\varepsilon(t)\rangle,
		\end{align}
		for a.e. $t\in[-t_0,T]$.
		Making use of H\"older's inequality, Sobolev embedding and Young's inequality, we get 
		\begin{align}\label{usc2}
		|\varepsilon b(\eta_\varepsilon,\z,\eta_\varepsilon)|&\leq\varepsilon\|\eta_\varepsilon\|^2_{\wi \L^4}\|\z\|_{\V}\leq \varepsilon\sqrt{2}\|\eta_\varepsilon\|_{\H}\|\eta_\varepsilon\|_{\V}\|\z\|_{\V}\leq\frac{\mu}{12}\|\eta_\varepsilon\|^2_{\V} +\varepsilon^2C\|\z\|^2_{\V}\|\eta_\varepsilon\|^2_{\H}, \\
		| b(\eta_\varepsilon,\u,\eta_\varepsilon)|&\leq\|\eta_\varepsilon\|^2_{\wi \L^4}\|\u\|_{\V}\leq \sqrt{2}\|\eta_\varepsilon\|_{\H}\|\eta_\varepsilon\|_{\V}\|\u\|_{\V}\leq\frac{\mu}{12}\|\eta_\varepsilon\|^2_{\V} +C\|\u\|^2_{\V}\|\eta_\varepsilon\|^2_{\H}, \\		
		|\varepsilon b(\z,\eta_\varepsilon,\z)|&\leq\varepsilon\|\z\|^2_{\wi \L^4}\|\eta_\varepsilon\|_{\V}\leq\varepsilon C\|\z\|^2_{\V}\|\eta_\varepsilon\|_{\V}\leq\frac{\mu}{12}\|\eta_\varepsilon\|^2_{\V} +\varepsilon^2C\|\z\|^4_{\V},\\
		|\varepsilon b(\z,\eta_\varepsilon,\u)|&\leq\varepsilon\|\z\|_{\wi \L^4}\|\u\|_{\wi \L^4}\|\eta_\varepsilon\|_{\V}\leq\varepsilon C\|\z\|_{\V}\|\u\|_{\V}\|\eta_\varepsilon\|_{\V}\leq\frac{\mu}{12}\|\eta_\varepsilon\|^2_{\V} +\varepsilon^2C\|\z\|^2_{\V}\|\u\|^2_{\V},\\
		|\varepsilon b(\u,\eta_\varepsilon,\z)|&\leq\varepsilon\|\u\|_{\wi \L^4}\|\z\|_{\wi \L^4}\|\eta_\varepsilon\|_{\V}\leq\varepsilon C\|\u\|_{\V}\|\z\|_{\V}\|\eta_\varepsilon\|_{\V}\leq\frac{\mu}{12}\|\eta_\varepsilon\|^2_{\V} +\varepsilon^2C\|\u\|^2_{\V}\|\z\|^2_{\V}.
		\end{align}
		By \eqref{MO_c}, we have 
		\begin{align}\label{usc3}
		-\beta\big\langle\mathcal{C}(\eta_\varepsilon+\varepsilon\z+\u)-\mathcal{C}(\u),\eta_\varepsilon+\varepsilon\z+\u-\u\big\rangle\leq 0.
		\end{align}
		Now, we consider
		\begin{align}
		\big|\varepsilon\beta\big\langle\mathcal{C}(\eta_\varepsilon+\varepsilon\z+\u)-\mathcal{C}(\u),\z\big\rangle\big|&\leq\big|\varepsilon\beta\big\langle\mathcal{C}(\u_\varepsilon),\z\big\rangle\big|+\big|\varepsilon\beta\big\langle\mathcal{C}(\u),\z\big\rangle\big|\nonumber\\&\leq\varepsilon\beta\big\{\|\u_\varepsilon\|^{r}_{\wi \L^{r+1}}+ \|\u\|^{r}_{\wi \L^{r+1}}\big\}\|\z\|_{\wi \L^{r+1}},\label{usc4}
		\end{align}
		and 
		\begin{align}
		\varepsilon\alpha\langle\z(t),\eta_\varepsilon(t)\rangle&\leq\varepsilon\alpha\|\z\|_{\H}\|\eta_\varepsilon\|_{\H}\leq\varepsilon\alpha C\|\z\|_{\V}\|\eta_\varepsilon\|_{\V}\leq\frac{\mu}{12}\|\eta_\varepsilon\|_{\V}^2+\varepsilon^2\alpha^2C\|\z\|^2_{\V}.\label{usc5}
		\end{align}
		Combining \eqref{usc2}-\eqref{usc5} and using in \eqref{usc1}, we deduce that 
		\begin{align*}
		\frac{\d}{\d t}\|\eta_\varepsilon(t)\|^2_{\H}&\leq C\big\{\varepsilon^2\|\z(t)\|^2_{\V}+\|\u(t)\|^2_{\V}\big\}\|\eta_\varepsilon(t)\|^2_{\H}+\varepsilon^2C\big\{\|\z(t)\|^2_{\V}+2\|\u(t)\|^2_{\V}+\alpha^2\big\}\|\z(t)\|^2_{\V}\\&\quad+\varepsilon\beta C\big\{\|\u_\varepsilon(t)\|^{r}_{\wi \L^{r+1}}+ \|\u(t)\|^{r}_{\wi \L^{r+1}}\big\}\|\z(t)\|_{\wi \L^{r+1}},\\
		&\leq C\big\{\varepsilon^2\|\z(t)\|^2_{\V}+\|\u(t)\|^2_{\V}\big\}\|\eta_\varepsilon(t)\|^2_{\H}+\varepsilon^2C\big\{\|\z(t)\|^2_{\V}+2\|\u(t)\|^2_{\V}+\alpha^2\big\}\|\z(t)\|^2_{\V}\\&\quad+\varepsilon\big\{\beta \|\u_\varepsilon(t)\|^{r+1}_{\wi \L^{r+1}}+ \beta\|\u(t)\|^{r+1}_{\wi \L^{r+1}}+ C\|\z(t)\|^{r+1}_{\wi \L^{r+1}}\big\},\\
		&\leq C\big\{\varepsilon^2\|\z(t)\|^2_{\V}+\|\u(t)\|^2_{\V}\big\}\|\eta_\varepsilon(t)\|^2_{\H}+\varepsilon^2C\big\{\|\z(t)\|^2_{\V}+2\|\u(t)\|^2_{\V}+\alpha^2\big\}\|\z(t)\|^2_{\V}\\&\quad+\varepsilon C\big\{\beta  \|\u_\varepsilon(t)\|^{r-1}_{\H}\|\u_{\varepsilon}(t)\|^2_{\V}+ \beta \|\u(t)\|^{r-1}_{\H}\|\u(t)\|^2_{\V}+ \|\z(t)\|^{r+1}_{\V}\big\}, 
		\end{align*}
		for a.e. $t\in[-t_0,T]$.	Now, integrating the above inequality from $-t_0$ to $t$, we obtain
		\begin{align}\label{usc6}
		\|\eta_\varepsilon(t)\|^2_{\H}\leq\|\eta_\varepsilon(-t_0)\|^2_{\H}+C\int_{-t_0}^{t}\upalpha_{\varepsilon}(s)\|\eta_\varepsilon(s)\|^2_{\H}\d s+\varepsilon C\int_{-t_0}^{t}\upbeta_{\varepsilon}(s)\d s,  \text{ for }t\in[-t_0,T],
		\end{align}
		where
		\begin{align*}
		\upalpha_{\varepsilon}=\ &\varepsilon^2\|\z\|^2_{\V}+\|\u\|^2_{\V}\\
		\upbeta_{\varepsilon}=\ &\varepsilon \big\{\|\z\|^2_{\V}+2\|\u\|^2_{\V}+\alpha^2\big\}\|\z\|^2_{\V}+\beta  \|\u_\varepsilon\|^{r-1}_{\H}\|\u_{\varepsilon}\|^2_{\V}+ \beta \|\u\|^{r-1}_{\H}\|\u\|^2_{\V}+ \|\z\|^{r+1}_{\V},
		\end{align*}
		for a.e. $t\in[-t_0,T]$. Then applying the Gronwall inequality, we find
		\begin{align}\label{usc7}
		\|\eta_\varepsilon(t)\|^2_{\H}\leq&\bigg(\varepsilon^2\|\z(-t_0)\|^2_{\H}+\varepsilon C\int_{-t_0}^{t}\upbeta_{\varepsilon}(s)\d s\bigg)e^{\int_{-t_0}^{t}\upalpha_{\varepsilon}(s)\d s}\nonumber\\
		\leq&\bigg(\frac{\varepsilon^2}{\lambda_1}\|\z(-t_0)\|^2_{\V}+\varepsilon C\int_{-t_0}^{t}\upbeta_{\varepsilon}(s)\d s\bigg)e^{\int_{-t_0}^{t}\upalpha_{\varepsilon}(s)\d s}.
		\end{align}
		Now, consider 
		\begin{align*}
		\int_{-t_0}^{t}\upalpha_{\varepsilon}(s)\d s&\leq  \varepsilon^2 (t+t_0)\|\z\|^2_{\mathrm{L}^{\infty}([-t_0,t];\V)}+\|\u\|_{\mathrm{L}^2{([-t_0,t];\V)}}, \\
		\int_{-t_0}^{t} \upbeta_{\varepsilon}(s)\d s&\leq  \varepsilon  \big\{(t+t_0)\|\z\|^2_{\mathrm{L}^{\infty}([-t_0,t]; \V)}+2\|\u\|_{\mathrm{L}^{2}([-t_0,t]; \V)}+(t+t_0) \alpha^2\big\} \|\z\|^2_{\mathrm{L}^{\infty}([-t_0,t]; \V)}\\& \quad+\beta \|\u_{\varepsilon}\|^{r-1}_{\mathrm{L}^{\infty}([-t_0,t];\H)}\|\u_{\varepsilon}\|_{\mathrm{L}^{2}([-t_0,t];\V)}+\beta \|\u\|^{r-1}_{\mathrm{L}^{\infty}([-t_0,t];\H)}\|\u\|_{\mathrm{L}^{2}([-t_0,t];\V)}\\& \quad+ (t+t_0)\|\z\|^{r+1}_{\mathrm{L}^{\infty}([-t_0,t]; \V)}.
		\end{align*}
		Since $\u_\varepsilon\text{ and } \u\in \mathrm{L}^{\infty}_{loc}([-t_0,\infty);\H)\cap\mathrm{L}^{2}_{loc}([-t_0,\infty);\V), \text{ and }\z\in\mathrm{L}^{\infty}_{\mathrm{loc}}([-t_0,\infty);\V)$,  therefore $\int_{-t_0}^{t}\upbeta_\varepsilon(s)\d s$ and $\int_{-t_0}^{t}\upalpha_\varepsilon(s)\d s$ both are finite. Hence, by \eqref{usc7}, we immediately have  $$\lim_{\varepsilon\to 0^+}\|\eta_\varepsilon(t)\|^2_{\H}=0,$$ which completes the proof of \eqref{usc} by taking $t=0$. Hence $(K_2)$ is verified.
		
		Since both $(K_1)$ and $(K_2)$ conditions hold for our model, the property of upper semicontinuity \eqref{51} holds true in $\H$.
		\vskip 2mm
		\noindent
		\textbf{Step III.} \emph{Proof of \eqref{K_3}:} In order to prove \eqref{K_3}, it is enough to prove that for any bounded subset $\G\subset \H,$ we have 
		\begin{align}\label{usc8}
		\lim_{\varepsilon\to\varepsilon_0}\sup_{\u_0\in \G} \|\varphi_{\varepsilon}(t_0, \theta_{-t_0}\omega)\u_0-\varphi_{\varepsilon_0}(t_0, \theta_{-t_0}\omega)\u_0\|_{\H}=0
		\end{align}
		For any $\u_0\in \G,$ let us take $\u_\varepsilon(t)=\varphi_{\varepsilon}(t+t_0,\theta_{-t_0})\u_0$ and $\u_{\varepsilon_0}(t)=\varphi_{\varepsilon_0}(t+t_0,\theta_{-t_0})\u_0$. Let $\u_{\varepsilon}(\cdot)$ be the unique weak solution of the system \eqref{S-CBF} and $\u_{\varepsilon_0}(\cdot)$ be the unique weak solution of the system \eqref{S-CBF} when $\varepsilon$ replaced by $\varepsilon_0$, with initial condition $\u_0$ at $t=-t_0$. Also, let $$\w(t)= \u_\varepsilon(t)-\u_{\varepsilon_0}(t), \ \ t\in[0,T].$$ 	Clearly $\w(\cdot)$ satisfies
		\begin{equation}\label{cSCBF_w}
		\left\{
		\begin{aligned}
		\d\w+\{\mu \A\w+\B(\w+\u_{\varepsilon_0})-\B(\u_{\varepsilon_0})+\beta \mathcal{C}(\w+\u_{\varepsilon_0})-\beta\mathcal{C}(\u_{\varepsilon_0})\}\d t&= \varepsilon^*\d\text{W}(t),  \\ 
		\w(-t_0)&=0,
		\end{aligned}
		\right.
		\end{equation}
		in $\V'$ for all $t\in[-t_0,T]$, where $\e^*=\e-\e_0$. 	Let us introduce $\varrho(\cdot)=\w(\cdot)-\varepsilon^*\z(\cdot)$, where $\z(\cdot)$ is the unique solution of \eqref{OUPe}. Then $\varrho(\cdot)$ satisfies the following equation in $\V'$:
		\begin{equation}\label{cscbf_rho2}
		\left\{
		\begin{aligned}
		\frac{\d\varrho}{\d t} &= -\mu \A\varrho - \B(\varrho,\varrho)-\varepsilon^*\B(\varrho,\z)-\B(\varrho,\u_{\varepsilon_0})-\varepsilon^*\B(\z,\varrho)-(\varepsilon^*)^2\B(\z,\z)-\varepsilon^*(\z,\u_{\varepsilon_0})\\&\quad-\B(\u_{\varepsilon_0},\varrho)-\varepsilon^*\B(\u_{\varepsilon_0},\z) - \beta \mathcal{C}(\varrho + \varepsilon^*\z +\u_{\varepsilon_0})+\beta\mathcal{C}(\u_{\varepsilon_0}) +\varepsilon^* \alpha \z, \\
		\varrho(-t_0)&= -\varepsilon^* \z(-t_0).
		\end{aligned}
		\right.
		\end{equation}
		The above system is similar to \eqref{cscbf_eta2} and a calculation similar to \eqref{usc6} yields 
		\begin{align}\label{usc14}
		\|\varrho(t)\|^2_{\H}\leq\|\varrho(-t_0)\|^2_{\H}+C\int_{-t_0}^{t}\textbf{p}_{\varepsilon}(s)\|\varrho(s)\|^2_{\H}\d s+\varepsilon^* C\int_{-t_0}^{t}\textbf{q}_{\varepsilon}(s)\d s,  \text{ for }t\in[-t_0,T],
		\end{align}
		where
		\begin{align*}
		\textbf{p}_{\varepsilon}=\ &(\varepsilon^*)^2\|\z\|^2_{\V}+\|\u_{\varepsilon_0}\|^2_{\V}\\
		\textbf{q}_{\varepsilon}=\ &\varepsilon^* \big\{\|\z\|^2_{\V}+2\|\u_{\varepsilon_0}\|^2_{\V}+\alpha^2\big\}\|\z\|^2_{\V}+\beta  \|\u_\varepsilon\|^{r-1}_{\H}\|\u_{\varepsilon}\|^2_{\V}+ \beta \|\u_{\varepsilon_0}\|^{r-1}_{\H}\|\u_{\varepsilon_0}\|^2_{\V}+ \|\z\|^{r+1}_{\V},
		\end{align*}
		for a.e. $t\in[-t_0,T]$. Then applying the Gronwall inequality, we deduce that 
		\begin{align}\label{usc15}
		\|\varrho(t)\|^2_{\H}
		\leq&\bigg\{(\varepsilon^*)^2C\|\z(-t_0)\|^2_{\V}+\varepsilon^* C\int_{-t_0}^{t}\textbf{q}_{\varepsilon}(s)\d s\bigg\}e^{\int_{-t_0}^{t}\textbf{p}_{\varepsilon}(s)\d s}.
		\end{align}
		Since $\u_\varepsilon,\u_{\varepsilon_0}\in \mathrm{L}^{\infty}_{loc}([-t_0,\infty);\H)\cap\mathrm{L}^{2}_{loc}([-t_0,\infty);\V), \text{ and }\z\in\mathrm{L}^{\infty}_{\mathrm{loc}}([-t_0,\infty);\V)$,  therefore $\int_{-t_0}^{t}\textbf{p}_\varepsilon(s)\d s$ and $\int_{-t_0}^{t}\textbf{q}_\varepsilon(s)\d s$ both are finite. Hence, by \eqref{usc15}, we immediately have  $$\lim_{\varepsilon\to \varepsilon_0}\|\varrho(t)\|^2_{\H}=0,$$which completes the proof of \eqref{usc8} by taking $t=0$.	Since \eqref{K_3} holds true, \eqref{52} follows immediately. 
	\end{proof}
	\begin{remark}\label{rem5.2}
		The upper semicontinuity of random attractors for non-compact random dynamical systems with an application to a stochastic reaction-diffusion equation on the whole space is discussed in \cite{BW}. Under the Assumption \ref{assump}, one can obtain the upper semicontinuity property of the random attractors for the 2D SCBF equations \eqref{S-CBF} in Poincar\'e domains also (see \cite{KM}), by using Theorem 3.1, \cite{BW} and following similarly as in Theorem \ref{USC}. 
	\end{remark}
	\section{Random attractors for 2D SCBF equations in $\V$}\label{sec6}\setcounter{equation}{0}
	In this section, we shall prove the existence of random attractors in more regular space $\V$. To prove the existence of random attractors in $\V$, we shall prove that our RDS satisfies pullback flattening property in $\V$. Here, we denote by $\hat{\mathfrak{DK}},$ the class of all closed and bounded random sets $\D(\omega)$ on $\V$ such that the radius function $\Omega\ni \omega \mapsto \kappa(\D(\omega)):= \sup\{\|x\|_{\V}:x\in \D(\omega)\}$ belongs to the class $\mathfrak{K}.$
	
	We have discussed in subsection \ref{Operator} that there exists an orthogonal basis $\{e_k\}_{k=1}^{\infty}$ of $\H$ consisting of eigenfunctions of $\A$ corresponding to the eigenvalues $\{\lambda_k\}_{k=1}^{\infty}$. Therefore, using an orthonormalization process, it is easy to see that the space $\H$ possesses an orthonormal basis $\{\phi_k\}^{\infty}_{k=1}$ of eigenfunctions of the Stokes operator $\A$ such that $$\A\phi_k=\lambda_k\phi_k.$$ Let us denote by $\H_m = \mathrm{span} \{\phi_1, \phi_2,\cdots, \phi_m\}$ and let $\mathrm{P}_m:\H\to \H_m$ be the $\H$ orthogonal projection onto $\H_m$.
	
	Since the existence of random absorbing sets in $\D(\A^s),s>1/2$ is not available for the 2D SCBF equations with small additive noise  \eqref{S-CBF}, the compactness arguments cannot be used to obtain the existence of random attractors in   $\mathbb{V}$. Therefore, in order to prove the existence of random attractors in $\V$, we prove that the cocycle $\varphi_\varepsilon$ corresponding to the 2D SCBF equations \eqref{S-CBF} satisfies the pullback flattening property in $\V.$
	\begin{theorem}\label{flattning}
		Suppose that $\f\in\H$ and Assumption \ref{assump} holds. Then for any $\varepsilon\in(0,1],$ the cocycle $\varphi_\varepsilon$ corresponding to the 2D SCBF equations with small additive noise \eqref{S-CBF} satisfies the pullback flattening property in $\V.$
	\end{theorem}
	\begin{proof}
		For given $s\leq0$, let $\hat{\B}=\{\B(\omega):\omega\in \Omega\}\in \hat{\mathfrak{DK}}, \ \delta>0,$ and  $$\v_\varepsilon(t) = \v_\varepsilon\big(t, s; \omega, \u_0 - \z(s)\big) = \varphi_\varepsilon(t-s; \theta_s \omega)\u_0 - \z(t)=\v_{\varepsilon,1}(t)+\v_{\varepsilon,2}(t),$$
		where $\u_0\in \B(\theta_s\omega),$ $\v_{\varepsilon,1}(t)=\mathrm{P}_m\v_\varepsilon(t), \v_{\varepsilon,2}(t)= \v_{\varepsilon}(t)-\mathrm{P}_m\v_\varepsilon(t)=\Q_m\v_\varepsilon(t).$
		Remember that for $\psi\in\D(\A)$, we have  \begin{align*}\mathrm{P}_m\psi&=\sum_{j=1}^{m}(\psi,\phi_j)\phi_j,\  \A\mathrm{P}_m\psi=\sum_{j=1}^{m}\lambda_j(\psi,\phi_j)\phi_j,\\ \mathrm{Q}_m\psi&=\sum_{j=m+1}^{\infty}(\psi,\phi_j)\phi_j,\ \A\mathrm{Q}_m\psi=\sum_{j=m+1}^{\infty}\lambda_j(\psi,\phi_j)\phi_j, \\
		\|\A\mathrm{Q}_m\psi\|_{\H}^2&=\sum_{j=m+1}^{\infty}\lambda_j^2|(\psi,\phi_j)|^2\geq \lambda_{m+1}\sum_{j=m+1}^{\infty}\lambda_j|(\psi,\phi_j)|^2=\lambda_{m+1}\|\Q_m\psi\|_{\V}^2,
		\end{align*}
		and 
		\begin{align*}
		\|\A\mathrm{P}_m\psi\|_{\H}^2=\sum_{j=1}^{m}\lambda_j^2|(\psi,\phi_j)|^2\leq \lambda_{m}\sum_{j=1}^{m}\lambda_j|(\psi,\phi_j)|^2=\lambda_{m}\|\P_m\psi\|_{\V}^2.
		\end{align*}
		That is, we get \begin{align}\label{4.9}\|\A\Q_m\psi\|_{\H}\geq \sqrt{\lambda_{m+1}}\|\Q_m\psi\|_{\V}\ \text{ and }\ \|\A\P_m\psi\|_{\H}\leq \sqrt{\lambda_m}\|\P_m\psi\|_{\V}.\end{align}
		Taking the inner product of the first equation of \eqref{cscbf} with $\A\v_{\varepsilon,2}(\cdot)$, we get
		\begin{align}\label{V_ab1_2}
		\frac{1}{2}\frac{\d}{\d t} \|\v_{\varepsilon,2}(t)\|^2_{\V}=&-\mu \|\A\v_{\varepsilon,2}(t)\|^2_{\H} -b(\v_\varepsilon(t)+\varepsilon\z(t),\v_\varepsilon(t)+\varepsilon\z(t),\A\v_{\varepsilon,2}(t))\nonumber\\&-\beta(\mathcal{C}(\v_\varepsilon(t)+\varepsilon\z(t)), \A\v_{\varepsilon,2}(t)) + \varepsilon\alpha(\z(t), \A\v_{\varepsilon,2}(t))+(\f,\A\v_{\varepsilon,2}(t)).
		\end{align}
		Using $0<\varepsilon\leq1$, \eqref{b1}, H\"older's and Young's inequalities, we obtain 
		\begin{align}
		\big|b(\v_\varepsilon+\varepsilon\z,\v_\varepsilon+\varepsilon\z,\A\v_{\varepsilon,2})\big|&\leq C \|\v_\varepsilon+\varepsilon\z\|^{\frac{1}{2}}_{\H}\|\v_\varepsilon+\varepsilon\z\|_{\V} \|\A\v_\varepsilon+\varepsilon\A\z\|^{\frac{1}{2}}_{\H}\|\A\v_{\varepsilon,2}\|_{\H} \nonumber\\
		&\leq\frac{\mu}{8}\|\A\v_{\varepsilon,2}\|^2_{\H} + C\|\v_\varepsilon+\varepsilon\z\|_{\H}\|\v_\varepsilon+\varepsilon\z\|^2_{\V}\|\A\v_\varepsilon+\varepsilon\A\z\|_{\H}\nonumber\\
		&\leq\frac{\mu}{8}\|\A\v_{\varepsilon,2}\|^2_{\H} + C\|\v_\varepsilon+\varepsilon\z\|^2_{\H}\|\v_\varepsilon+\varepsilon\z\|^4_{\V}+C\|\A\v_\varepsilon\|^2_{\H}+ C\|\A\z\|^2_{\H}\label{V_ab2_2},\\
		\big|(\mathcal{C}\big(\v_\varepsilon+\varepsilon\z\big),\A\v_{\varepsilon,2})\big|&\leq \|\v_\varepsilon+\varepsilon\z\|^{r}_{\wi \L^{2r}}\|\A\v_{\varepsilon,2}\|_{\H}\leq \frac{\mu}{8}\|\A\v_{\varepsilon,2}\|^2_{\H} + C\|\v_\varepsilon+\varepsilon\z\|^{2r}_{\wi \L^{2r}}\nonumber\\&\leq \frac{\mu}{8}\|\A\v_{\varepsilon,2}\|^2_{\H} + C\|\v_\varepsilon+\varepsilon\z\|^{2r-2}_{\H}\|\v_\varepsilon+\varepsilon\z\|^2_{\V},\label{V_ab3_2}\\
		\varepsilon\alpha\big(\z,\A\v_{\varepsilon,2}\big)&\leq \alpha \|\z\|_{\H}\|\A\v_{\varepsilon,2}\|_{\H}	\leq\frac{\mu}{8}\|\A\v_{\varepsilon,2}\|^2_{\H} + C\|\z\|^2_{\H}\leq\frac{\mu}{8}\|\A\v_{\varepsilon,2}\|^2_{\H} + C\|\z\|^2_{\V},\label{V_ab4_2}\\
		\big(\f, \A\v_{\varepsilon,2}\big)&\leq \|\f\|_{\H}\|\A\v_{\varepsilon,2}\|_{\H}\leq\frac{\mu}{8}\|\A\v_{\varepsilon,2}\|^2_{\H} + C\|\f\|^2_{\H}.\label{V_ab5_2}
		\end{align}
		Using the inequalities \eqref{V_ab2_2}-\eqref{V_ab5_2} in \eqref{V_ab1_2} and then using \eqref{4.9}, we deduce that for any $\varepsilon\in(0,1]$ and for any $\omega\in\Omega$,
		\begin{align*}
		\frac{\d}{\d t}\|\v_{\varepsilon,2}(t)\|^2_{\V} +\mu\lambda_{m+1}\|\v_{\varepsilon,2}(t)\|^2_{\V} 
		\leq&\ C\|\v_\varepsilon(t)+\varepsilon\z(t)\|^2_{\H}\|\v_\varepsilon(t)+\varepsilon\z(t)\|^4_{\V} +C\|\A\v_\varepsilon(t)\|^2_{\H} \nonumber\\&+ C\|\A\z(t)\|^2_{\H}+C\|\v_\varepsilon(t)+\varepsilon\z(t)\|^{2 r-2}_{\H}\|\v_\varepsilon(t)+\varepsilon\z(t)\|^2_{\V}\nonumber\\&+C\|\z(t)\|^2_{\V}+C\|\f\|^2_{\H}.
		\end{align*}
		Thus, it is immediate that 
		\begin{align}\label{V_ab_7_2}
		\frac{\d}{\d t}\left[e^{\mu\lambda_{m+1} t}\|\v_{\varepsilon,2}(t)\|^2_{\V} \right]
		\leq&\ \bigg[C\|\v_\varepsilon(t)+\varepsilon\z(t)\|^2_{\H}\|\v_\varepsilon(t)+\varepsilon\z(t)\|^4_{\V} +C\|\A\v_\varepsilon(t)\|^2_{\H} \nonumber\\&+ C\|\A\z(t)\|^2_{\H}+C\|\v_\varepsilon(t)+\varepsilon\z(t)\|^{2 r-2}_{\H}\|\v_\varepsilon(t)+\varepsilon\z(t)\|^2_{\V}\nonumber\\&+C\|\z(t)\|^2_{\V}+C\|\f\|^2_{\H}\bigg]e^{\mu\lambda_{m+1} t}.
		\end{align}
		From Theorems \ref{H_ab} and \ref{V_ab}, we get for any $\varepsilon\in(0,1]$ and for any $t\in[-1,0]$, there exists $\kappa_{17}(\omega)\geq0$ and $\kappa_{18}(\omega)\geq0$ such that
		\begin{align}
		\|\v_\varepsilon(t,s;\omega,\u_0-\z(s))\|_{\H}\leq\kappa_{17}(\omega) \quad\text{and}\quad
		\|\v_\varepsilon(t,s;\omega,\u_0-\z(s))\|_{\V}\leq\kappa_{18}(\omega),\label{ab}
		\end{align}
		for any $s\leq-(t_{\mathrm{D}}(\omega)+2).$
		Hence by an  application of the uniform Gronwall Lemma, using \eqref{V_ab_12} and \eqref{ab} in \eqref{V_ab_7_2}, we deduce that for any $\varepsilon\in(0,1]$ and for any $\omega\in\Omega,$ there exists $\kappa_{19}(\omega)\geq0$  such that 
		\begin{align}
		\|\v_{\varepsilon,2}(0,\omega; s, \u_0-\z(s))\|_{\V}\leq \kappa_{19}(\omega)e^{-\mu\lambda_{m+1}},
		\end{align}
		for any $s\leq-\big(t_{\mathrm{D}}(\omega)+2\big).$
		Therefore, for sufficiently large $m$, we get
		$$\|\Q_m\v_\varepsilon(0,\omega; s, \u_0-\z(s))\|^2_{\V}\leq\delta,$$
		for any $\delta>0, \omega\in \Omega$ and any $s\leq-\big(t_{\mathrm{D}}(\omega)+2\big).$
	\end{proof}
	
	From Theorems \ref{H_ab}, \ref{V_ab}, \ref{flattning} and \ref{flat_T}, we immediately conclude the following result:
	\begin{theorem}\label{Main_Theoem_V}
		Suppose that $\f\in\H$ and Assumption \ref{assump} holds. Then for any $\varepsilon\in(0,1],$ the cocycle $\varphi_\varepsilon$ corresponding to 2D SCBF equations with small additive noise \eqref{S-CBF} has a $\hat{\mathfrak{DK}}$-random attractor $\mathcal{G}_\varepsilon = \{\textbf{G}_\varepsilon(\omega): \omega\in \Omega\}$ in $\V$. 
	\end{theorem}
	\section{Invariant measures}\label{sec7}\setcounter{equation}{0}
	In this section, we discuss about the existence of an invariant measure for the 2D SCBF equations \eqref{S-CBF}, which is a direct consequence of Corollary 4.4, \cite{CF} along with Theorems \ref{Main_Theoem_H} and  \ref{Main_Theoem_V}.
	
	Let $\varphi_{\e}$ be the RDS corresponding to the 2D SCBF equations \eqref{S-CBF}, which is defined by \eqref{combine_sol}. Let us define the transition operator $\mathrm{P}_t$ by \begin{align}\label{71}\mathrm{P}_t f(\x)=\int_{\Omega}f(\varphi_{\e}(t,\omega,\x))\d\mathbb{P}(\omega)=\E\left[f(\varphi_{\e}(t,\x))\right],\end{align}  for all $f\in\mathcal{B}_b(\H)$, where $\mathcal{B}_b(\H)$ is the space of all bounded and Borel measurable functions on $\H$. A proof similar to Proposition 3.8, \cite{BL} yields the following result: 
	\begin{lemma}
		The family $\{\mathrm{P}_t\}_{t\geq 0}$ is Feller, that is, $\mathrm{P}_tf\in\C_{b}(\H)$ if $f\in\C_b(\H)$, where $\C_b(\H)$ is the space of all bounded and continuous functions on $\H$. Furthermore, for any $f\in\C_b(\H)$, $\mathrm{P}_tf(\x)\to f(\x)$ as $t\downarrow 0$. 
	\end{lemma}
	Using similar arguments as in the proof of Theorem 5.6, \cite{CF}, one can prove that $\varphi_{\e}$ is a Markov RDS, that is, $\mathrm{P}_{t+s}=\mathrm{P}_t\mathrm{P}_s$, for all $t,s\geq 0$. 
	Since, we know by Corollary 4.4, \cite{CF} that if a Markov RDS on a Polish space has an invariant compact random set, then there exists a Feller invariant probability measure $\nu_{\e}$ for $\varphi_{\e}$. 
	\begin{definition}
		A Borel probability measure $\nu$ on $\H$  is called an invariant measure	for a Markov semigroup $\{\mathrm{P}_t\}_{t\geq 0}$ of Feller operators on $\C_b(\H)$ if and only if $$\mathrm{P}_t^*\nu=\nu, \ t\geq 0,$$ where $(\mathrm{P}_t^*\nu)(\Gamma)=\int_{\mathbb{H}}\mathrm{P}_t(\x,\Gamma)\nu(\d\x)$ for $\Gamma\in\mathcal{B}(\H)$ and the $\mathrm{P}_t(\x,\cdot)$ is the transition probability, $\mathrm{P}_t(\x,\Gamma)=\mathrm{P}_t(\chi_{\Gamma})(\x),\ \x\in\H$.
	\end{definition}
	
	In Theorems \ref{Main_Theoem_H} and  \ref{Main_Theoem_V}, we have proved the existence of random attractors in $\H$ and in $\V$, respectively. By the definition of random attractors, it is immediate  that there exists an invariant compact random set in $\H$ as well as in $\V$. A Feller invariant probability measure for a Markov RDS $\varphi$ on $\H$ is, by definition, an invariant probability measure for the semigroup $\{\mathrm{P}_t\}_{t\geq 0}$ defined by \eqref{71}. Hence, we have the following result on the existence of invariant measures for the 2D SCBF equations \eqref{S-CBF}.
	\begin{theorem}\label{thm7.3}
		There exists an invariant measure for the 2D SCBF equations \eqref{S-CBF} in $\H$.
	\end{theorem}
	
	\begin{remark}
		1.	In  Theorem \ref{Main_Theoem_V}, we have also proved that there exists a random attractor in $\V$ and hence there exists an invariant compact random set in $\V$. Invoking Corollary 4.4 in \cite{CF}, we obtain the existence of an invariant measure for the 2D SCBF equations \eqref{S-CBF} in $\V$ as well.
		
		2. The uniqueness of invariant measures for the  SCBF equations by using the exponential stability results has been established in Theorem 5.5, \cite{MTM}. 
	\end{remark}

	\medskip\noindent
	{\bf Acknowledgments:}    The first author would like to thank the Council of Scientific $\&$ Industrial Research (CSIR), India for financial assistance (File No. 09/143(0938)/2019-EMR-I).  M. T. Mohan would  like to thank the Department of Science and Technology (DST), Govt of India for Innovation in Science Pursuit for Inspired Research (INSPIRE) Faculty Award (IFA17-MA110).

\end{document}